\documentclass[reqno]{amsart}
\usepackage{amssymb,amscd,verbatim, amsthm,graphics, color,latexsym,amsmath,multicol}
\usepackage{fancybox}
\usepackage{longtable}
\usepackage{rotating}

\usepackage[all]{xy}

\newcommand \fk[1]{{{\mathfrak #1}}}
\newcommand \C[1]{{\mathcal #1}}

\newcommand \wti[1]{{\widetilde {#1}}}

\newcommand \bC{{\mathbb C}}
\newcommand \bF{{\mathbb F}}
\newcommand \bH{{\mathbb H}}

\newcommand \bZ{{\mathbb Z}}

\newcommand \bQ{{\mathbb Q}}

\newcommand \bbq{{\mathsf {q}}}

\newcommand\CA{{\C A}}

\newcommand\CH{{\C H}}
\newcommand\CI{{\C I}}
\newcommand\CO{{\C O}}

\newcommand\al{{\alpha}}

\newtheorem{theorem}{Theorem}[section]

\newtheorem{corollary}[theorem]{Corollary}
\newtheorem{lemma}[theorem]{Lemma}
\newtheorem{proposition}[theorem]{Proposition}

\theoremstyle{definition}
\newtheorem{definition}[theorem]{Definition}
\newtheorem{remark}[theorem]{Remark}
\newtheorem{example}[theorem]{Example}

\newcommand\Hom{\operatorname{Hom}}
\newcommand\End{\operatorname{End}}
\newcommand\Ind{\operatorname{Ind}}
\newcommand\rank{\operatorname{rank}}
\newcommand\tr{\operatorname{tr}}

\newcommand\im{\operatorname{im}}

\newcommand\St{\mathsf{St}}
\newcommand\sem{\mathsf{ss}}
\newcommand\Irr{\mathsf{Irr}}
\newcommand\good{\mathsf{good}}

\newcommand\twist{\mathsf{diff-ind}}

\newcommand\rig{\mathsf{rigid}}
\newcommand\Rk{\mathsf{Rk}}

\newcommand\EL{\mathsf{ell}}

\def\<{\langle} 
\def\>{\rangle}
\def\i{^{-1}}
\def\s{\sigma}

\def\co{\mathcal O}
\def\d{\delta}

\newcommand{\remind}[1]{{\bf \EL** #1 **}}

\numberwithin{equation}{section}

\begin{document}

\title{Cocenters and representations of affine Hecke algebra}

\author{Dan Ciubotaru}
      \address[D.Ciubotaru]{Mathematical Institute\\University of Oxford\\Oxford\\OX2 6GG\\UK}
      \email{dan.ciubotaru@maths.ox.ac.uk}
      
\author{Xuhua He}
\address[X. He]{Department of Mathematics, University of Maryland, College Park, MD 20742, USA and Department of Mathematics, HKUST, Hong Kong}
\email{xuhuahe@math.umd.edu}

\thanks{The authors thank R. Bezrukavnikov, S. Kato, R. Kottwitz, G. Lusztig, E. Opdam, M. F. Vign\'eras, D. Vogan and N. Xi for useful comments about this paper. D. C. was supported in part by NSF-DMS 1302122 and NSA-AMS 111016, and X. H. was supported in part by HKRGC grant 602011.}



\begin{abstract}In this paper, we study the relation between the cocenter and the representation theory of affine Hecke algebras. The approach is based on the interaction between the rigid cocenter, an important subspace of the cocenter, and the dual object in representation theory, the rigid quotient of the Grothendieck group of finite dimensional representations.
\end{abstract}

\maketitle

\section{Introduction}\label{sec:1}
Affine Hecke algebras appear naturally in the representation theory of reductive $p$-adic groups as convolution algebras of compactly supported functions, such as Iwahori-Hecke algebras e.g., \cite{Bo},\cite{IM}, and their generalizations in the theory of types, e.g., \cite{HM},\cite{BK}, or as endomorphism algebras of certain projective generators \cite{Be},\cite{He}. The representation theory of affine Hecke algebras with (equal and unequal) parameters that are not roots of unity was extensively studied, e.g., \cite{KL},\cite{CG},\cite{L1,L2},\cite{BM1,BM2}, \cite{Xi}, \cite{Op,OS1,OS2,So},\cite{Re},\cite{Kat}. The representation theory of affine Hecke algebras for parameter equal to a root of unity plays an important role in the study of modular representations of $p$-adic groups \cite{Vi} and the simple modular representations for affine Hecke algebras associated to general linear groups were classified in \cite{AM}. The representations of affine Hecke algebras of classical types at roots of unity were subsequently studied in \cite{VV,SVV} via the theory of canonical bases, and for $G_2$ in \cite{Xi2}. 

In this paper, we consider finite dimensional representations of the affine Hecke algebra $\CH$ from the perspective of its relation with the cocenter of the algebra. The cocenter $\bar\CH$ is the quotient of the algebra by the subspace of commutators, and thus appears naturally in duality (given by the trace) with the (complexification of the) Grothendieck group $R(\CH)_\bC$ of finite dimensional $\CH$-modules.  

\subsection{} We explain our main results. In the body of the paper, we consider the more general case of an affine Hecke algebra with an automorphism $\delta$ of the root system and the $\delta$-twisted cocenter, but, for simplicity, in the introduction, we present the results in the untwisted case only. We also consider arbitrary {(nonzero)} parameters.

{The affine Hecke algebra $\CH$ is a deformation of the group algebra $\bC[\wti W]$ for the extended affine Weyl group $\wti W$. Let $cl(\wti W)$ denote the set of conjugacy classes in $\wti W$. It is easy to see that for any two elements $w, w'$ in the same conjugacy class of $\wti W$, the images of $w$ and $w'$ in the cocenter of $\bC[\wti W]$ are the same and the set $\{[\co]; \co \in cl(\wti W)]\}$ is a basis of the cocenter of $\bC[\wti W]$. Here $[\co]$ is the image of $w$ in the cocenter of $\bC[\wti W]$ for any $w \in \co$. 

However, for $w, w'$ in a conjugacy class $\co$ of $\wti W$, the images in $\bar \CH$ of the standard basis elements $T_w$ and $T_{w'}$ of $\CH$ are not the same in general. It is showed in \cite{HN} that if $w$ and $w'$ are of minimal length in $\co$, then the image of $T_w$ and $T_{w'}$ in $\bar \CH$ are the same. We denote the image by $T_\co$. Moreover, the whole cocenter $\bar \CH$ is spanned by $\{T_\co; \co \in cl(\wti W)\}$. 

Note that $cl(\wti W)$ is a countable infinite set and the set of irreducible representations of $\CH$ is an uncountable infinite set. To compare the cocenter with representations, we would like to develop a reduction method from infinite sets to finite sets. 

A familiar object in the literature is the elliptic quotient $\bar R_0(\CH)$ obtained by taking the quotient of $R(\CH)_\bC$ by the span of all proper parabolically induced modules \cite{BDK}. The elliptic representation theory of reductive $p$-adic groups and associated Hecke algebras has been an area of active research, e.g., \cite{Ar}, \cite{BDK}, \cite{Bez}, \cite{Kaz}, \cite{SS}, \cite{Re}, \cite{OS2}. The dual object to $\bar R_0(\CH)$ is the elliptic cocenter $\bar\CH^{\EL}$,  the subspace of $\bar\CH$ on which all proper parabolically induced modules vanish. However, as shown in \cite{BDK}, the elliptic cocenter is very complicated to understand. 
}

{\subsection{} The solution we provide in this paper is another quotient of $R(\CH)_\bC$ and another subspace of $\bar \CH$, which we call the rigid quotient and rigid cocenter, respectively. 

Let us first describe our motivation leading to the definitions and then describe the main result and its consequences. 

For simplicity, we only consider the affine Hecke algebras associated to semisimple root datum in this introduction. In the beginning, we would like to get a nice finite subset of $cl(\wti W)$. A natural choice is $cl(\wti W)_{\EL}$, the set of elliptic conjugacy classes of $\wti W$. The problem is that there is no relation between the subspace spanned by $\{T_\co; \co \in cl(\wti W)_{\EL}\}$ and the elliptic cocenter. 

The finite subset of $cl(\wti W)$ we use here is $cl(\wti W)_0$, the conjugacy classes of $\wti W$ with zero Newton point. We have the inclusion $cl(\wti W)_{\EL} \subset cl(\wti W)_0 \subset cl(\wti W)$. The definition of $cl(\wti W)_0$ and the idea to use it in the study of affine Hecke algebras are inspired by a classical result in arithmetic geometry: Kottwitz's classification of $\s$-isocrystals, \cite{Ko1,Ko2}. 

Here we provide more content. While this is not needed in our theory of rigid cocenter and rigid quotient, it serves as motivations for it. 

Let $G$ be a connected reductive group split over $\bQ_p$ and $L=W(\bar \bF_p)[\frac{1}{p}]$ be the completion of the maximal unramified extension of $\bQ_p$. Let $\s$ be the (relative) Frobenius morphism on $G(L)$ and $B(G)$ be the set of $\s$-conjugacy classes of $G(L)$. Kottwitz showed that a $\s$-conjugacy class is determined by two invariants given by the image of the Kottwitz map $B(G) \to \pi(G)$ and the Newton map from $B(G)$ to the rational coweight lattice. A $\s$-conjugacy class is called basic if its Newton point is zero. Moreover, any $\s$-conjugacy class comes from a basic $\s$-conjugacy class of some Levi subgroup $M(L)$ of $G(L)$ via the inclusion map $M \to G$. 

In \cite{He99}, one studied the natural map $cl(\wti W) \to B(G)$. This map is finite-to-one and it is compatible with both the Kottwitz map and the Newton map. Thus we have a Cartesian diagram \[\xymatrix{cl(\wti W)_0 \ar[r] \ar[d] & B(G)_{basic} \ar[d] \\ cl(\wti W) \ar[r] & B(G),}\] where $B(G)_{basic}$ is the set of basic $\s$-conjugacy classes. 

As the basic $\s$-conjugacy classes capture the whole $B(G)$ in an essential way, we expect that $cl(\wti W)_0$ plays an essential role in the study of $\bar \CH$. This leads to the following definition of rigid cocenter: $$\bar \CH^\rig:=\text{span}\{T_\co; \co \in cl(\wti W)_0\}.$$

It turns out that for generic parameters, the quotient of $R(\CH)_\bC$ dual to $\bar \CH^\rig$ is also a natural object. Namely, $\bar R(\CH)_\rig$ is defined as the quotient of $R(\CH)_\bC$ by $R(\CH)_\twist$, the span  of differences of central twists of parabolically induced modules, see Definition \ref{d:Hecke-rigid}. 
}

\smallskip

The main result concerning the rigid cocenter is the following theorem.

\begin{theorem}\label{t:main-rigid}
\begin{enumerate}
\item The set $\{T_\CO:\CO\in cl(\wti W)_0\}$ is a basis of $\bar \CH^\rig$ for an affine Hecke algebra $\CH$ with arbitrary parameters.
\item Suppose that the parameters of the affine Hecke algebra are admissible in the sense of Definition \ref{d:generic}. Then the trace pairing $\tr:\bar \CH\times  R(\CH)_\bC\to\bC$ induces a perfect pairing $$\tr:\bar \CH^\rig\times \bar R(\CH)_\rig\to\bC.$$
In particular, the dimension of $\bar R(\CH)_\rig$ equals the number of classes in $cl(\wti W)_0.$
\item For arbitrary parameters, the trace map $\tr:\bar \CH^\rig\to R(\CH)_\rig^*$ is surjective.
\end{enumerate}
\end{theorem}
{Theorem \ref{t:main-rigid} in particular explains where the name ``rigid'' for $\bar \CH^\rig$ comes from:
\begin{itemize}
\item The traces of any given element in $\bar \CH^\rig$ on the parabolically induced modules are constant if we deform the central characters and the image under the trace map of $\bar \CH^\rig$ gives all such ``rigid'' linear functions on the induced modules. 
\item The rigid cocenter $\bar \CH^\rig$ has a basis which is independent of the parameters of the Hecke algebra and thus $\bar \CH^\rig$ is ``rigid'' if we deform the parameters. 

\end{itemize}
}

As far as the relation between the rigid and elliptic cocenters/quotients, clearly, $\bar\CH^{\EL}\subset\bar\CH^\rig$ and $\bar R(\CH)_\rig\twoheadrightarrow\bar R_0(\CH).$ But as we show in the paper, see Proposition \ref{p:rigid-cocenter}, the rigid cocenter, in fact, combines together the elliptic cocenters of all the semisimple parts of parabolic subalgebras, up to a certain equivalence. Dually, the rigid quotient admits a section formed by the elliptic quotients of the semisimple parts of parabolic subalgebras (up to equivalence), see Corollary \ref{c:rigid-section}. Thus the rigid cocenter/quotient allows us to study the elliptic theory for all the parabolic subalgebras at once. 

\smallskip

{As a consequence, we obtain the basis theorem of $\bar \CH$.

\begin{theorem}
The set $\{T_\CO:\CO\in cl(\wti W)\}$ is a basis of $\bar \CH$ for an affine Hecke algebra $\CH$ with arbitrary parameters.
\end{theorem}

This theorem in particular shows that the cocenter of $\CH$ has a description independent of the parameters. For finite Hecke algebras, Tits' deformation theorem says that for generic parameters, the finite Hecke algebra is isomorphic to the group algebra. However, affine Hecke algebras of the same type but different parameters are almost never isomorphic. The basis theorems for $\bar \CH$ and $\bar \CH^\rig$ above, in some sense, provide a substitute for the Tits' deformation theorem in the affine setting. 
}

\subsection{}As consequences of our approach, we obtain direct, algebraic proofs of the analogues of classical results from $p$-adic groups: the Density Theorem and the trace Paley-Wiener Theorem. For $p$-adic groups, proofs of these results are known from \cite{BDK} and \cite{Ka}, see also \cite{Da} and \cite{Fl}. For affine Hecke algebras with positive parameters, the Density Theorem and trace Paley-Wiener theorem are studied in \cite[Theorem 3.4]{So} by different methods.

We summarize our results in the next theorem.

\begin{theorem}\label{t:main-density}

\begin{enumerate}
\item (Density Theorem) Suppose that the parameters of the affine Hecke algebra are admissible. Then 
the trace map $\tr:\bar \CH\to R(\CH)^*$ is injective.
\item (Trace Paley-Wiener Theorem) For arbitrary parameters, the image of the trace map $\tr:\bar \CH\to R(\CH)^*$ is the space $R^*(\CH)_\good$ of good forms (Definition \ref{d:goodforms}).
\end{enumerate}
\end{theorem}

{We would like to point out that the density theorem fails for affine Hecke algebras at some roots of unity but the trace Paley-Wiener theorem always holds. The trace Paley-Wiener Theorem for affine Hecke algebras at roots of unity seems to be a new result and will play a role in the modular representations of $p$-adic groups. }

\

\subsection{}
As a different application, in section \ref{s:unitary}, we show how Theorem \ref{t:main-rigid} can be used to simplify the arguments of \cite{BM1,BM2} for the preservation of unitarity under the Borel functor for the category of smooth representations with Iwahori fixed vectors of a semisimple $p$-adic group $G$. 
Let $I$ be an Iwahori subgroup and if $H\subset G$ is a closed subgroup containing $I$, let $\CH(H//I)$ denote the Iwahori-Hecke algebra of $I$-biinvariant compactly supported functions with support in $H$. The Borel-Casselman correspondence says that the functor $V\to V^I$ is an equivalence between the subcategory of smooth complex $G$-representations generated by their $I$-fixed vectors and the category of $\CH(G//I)$-modules. The main result of \cite{BM1,BM2}, see also \cite{BC},  is the following theorem conjectured by Borel.

\begin{theorem}[Barbasch-Moy preservation of unitarity]\label{t:BM-intro} Under the functor $V\to V^I$, the $G$-representation $V$ is unitary if and only if the $\CH(G//I)$-module $V^I$ is unitary.
\end{theorem}

The classical arguments of \cite{BM1,BM2} use two main ingredients. The first is an analogue of Vogan's signature character \cite{Vo} which expresses the signature of an irreducible hermitian representation in terms of the $K$-character ($K$ a maximal compact open subgroup) of tempered modules. The second ingredient is a linear independence of the $\CH(K_0//I)$-characters ($K_0$ hyperspecial) of irreducible tempered representations with "real infinitesimal character" combined with a subtle reduction to real infinitesimal character in \cite{BM2} via endoscopic groups. The rigid quotient provides for a simplification of the argument (without the need for the reduction to real infinitesimal character) by: firstly, considering the $K$-characters for all maximal parahoric subgroups, and, secondly, the fact that, by Theorem \ref{t:main-rigid}(2), the characters of a basis of the rigid quotient (the basis can be chosen to consist of tempered representations) are linearly independent when restricted to the union of $\CH(K//I)$'s.

\subsection{} We give an outline of the paper. In sections \ref{sec:2} and \ref{sec:clifford}, we establish the notation and give the basic definitions for the affine Hecke algebra (extended by outer automorphisms), the cocenter and $R(\CH)$. In section \ref{sec:3}, we recall in our setting the main definitions and properties of the induction and restriction functors, as introduced for $p$-adic groups in \cite{BDK} and studied further in \cite{Da}. In section \ref{sec:span-cocenter}, we record the main  cocenter results from \cite{HN,HN2} that we will use in the rest of the paper. In sections \ref{s:6} and \ref{s:7}, we define the rigid and elliptic cocenters and the rigid and elliptic quotients of $R(\CH)$, and prove the main results about the duality between the rigid cocenter and rigid quotient, in particular, Theorem \ref{t:main-rigid}. Some of these results are proven under the assumption that (PDT), the density theorem for parabolic subalgebras, holds. In section \ref{s:8}, we prove (by induction) the Basis Theorem, Density Theorem, and Trace Paley-Wiener Theorem, as enumerated in Theorem \ref{t:main-density}. In particular, we now see that the assumption (PDT) can be removed from the previous results. We also obtain sharp bounds for the dimension of the elliptic quotient for arbitrary parameters. Finally, in section \ref{s:unitary}, we give our application to the preservation of unitarity argument.

\section{Preliminaries}\label{sec:2}

\subsection{} Let $\Phi=(X,R,X^\vee,R^\vee,\Pi)$ be a reduced based
root datum. In particular,
\begin{enumerate}
\item[(a)] $X,X^\vee$ are free abelian groups of finite rank with a
  perfect pairing $\langle ~,~\rangle:X\times X^\vee\to \bZ$;
\item[(b)] $R\subset X$, $R^\vee\subset X^\vee$ are the roots and
  coroots, respectively, in bijection $\alpha\leftrightarrow
  \al^\vee$;
\item[(c)] $\Pi\subset R$ is the set of simple roots.
\end{enumerate}
Let $R^+$ denote the positive roots defined by $\Pi$, and
$\Pi^-=-\Pi^+$; let $R^{\vee,+}$, $R^{\vee,-}$ denote the
corresponding coroots. For every $\alpha\in R,$ let $s_\al\in GL(X)$
be the reflection $s_\al(x)=x-\langle x,\al^\vee\rangle\al.$ Let
$W\subset GL(X)$
be the finite Weyl group, generated by $S=\{s_\al:\al\in\Pi\}.$

\subsection{} The set of affine roots is $R^a=R^\vee\times\bZ.$ Let
$\le$ be the partial order of $R^\vee$ defined by $\beta^\vee\le
\al^\vee$ is $\al^\vee-\beta^\vee$ is a nonnegative integer linear
combination of $\{\al^\vee:\al\in\Pi\}.$ Set $R_m=\{\gamma\in R:
\gamma^\vee \text{ is minimal in } (R^\vee,\le)\}.$ The simple affine
roots are
\begin{equation}
\Pi^a=\{(\al,0): \al\in\Pi\}\cup \{(\gamma^\vee,1): \gamma\in R_m\}. 
\end{equation}
Define the extended affine Weyl group:
\begin{equation}
\wti W=X\rtimes W.
\end{equation}
We write a typical element in $\wti W$ as $t_x w$, $x\in X,$ $w\in
W$. The multiplication is then $(t_x w)\cdot (t_{x'} w')=t_{x+w(x')}
ww'.$ The group $\wti W$ acts on $X$ by $(t_xw)\cdot y=x+w(y).$

Define also the affine Weyl group:
\begin{equation}
W^a=Q\rtimes W,
\end{equation}
where $Q\subset X$ is the root lattice, i.e., the $\bZ$-span of
$R$. The group $W^a$ is an infinite Coxeter group generated by
$S^a=S\cup \{t_{-\gamma} s_\gamma: \gamma\in R_m\}.$ Define positive
and negative affine roots as follows:
\begin{equation}
\begin{aligned}
R^{a,+}&=(R^\vee\times \bZ_{>0})\cup (R^{\vee,+}\times\{0\}),\\
R^{a,-}&=(R^\vee\times \bZ_{<0})\cup (R^{\vee,-}\times\{0\}),\\
\end{aligned}
\end{equation} 
and the length function $\ell:\wti W\to \bZ_{\ge 0}$,
\begin{equation}
\ell(w)=\#\{\al_a\in R^{a,+}: w\al_a\in R^{a,-}\}, \quad w\in \wti W.
\end{equation}
Set $\Omega=\{w\in \wti W: \ell(w)=0\}.$ Then $\wti W=W^a\rtimes
\Omega$, and $\Omega\cong X/Q.$

\subsection{} Fix a set of indeterminates $\bbq=\{\bbq(s): s\in S^a\}$
such that $\bbq(ws)=\bbq(s)$ for all $w\in \wti W$, and let
$\Lambda=\bC[\bbq(s)^{\pm 1}: s\in S^a].$

\begin{definition}[Iwahori-Matsumoto presentation]\label{d:IM} The affine Hecke algebra
  $\CH=\CH(\Phi,\bbq)$ is the $\Lambda$-algebra generated by $\{T_w:
  w\in \wti W\}$ subject to the relations:
\begin{enumerate}
\item $T_w\cdot T_{w'}=T_{ww'}$, if $\ell(ww')=\ell(w)+\ell(w')$;
\item $(T_s+1)(T_s-\bbq(s)^2)=0,$ $s\in S^a$.
\end{enumerate}
\end{definition}

The algebra $\CH(\Phi,\bbq)$ admits a second presentation, due to
Bernstein and Lusztig, that we recall next. If $w\in W^a$ has a
reduced expression $w=s_1\cdots s_k,$ $s_i\in S^a$, set $\bbq(w)=\prod
\bbq(s_i).$ Extend this further to $\wti W$ by setting $\bbq(u)=1$ for
all $u\in\Omega.$

Define $X_+=\{x\in X: \langle x,\al^\vee\rangle\ge 0,\text{ for all
}\al\in\Pi\}.$ If $x\in X,$ write $x=x_1-x_2,$ $x_1,x_2\in
X_+$. Define
\begin{equation}
\theta_x=\bbq(t_{x_1})^{-1}\bbq(t_{x_2}) T_{t_{x_1}} T_{t_{x_2}}^{-1}.
\end{equation}
Then $\{T_w\theta_x: w\in W, x\in X\}$ forms a $\Lambda$-basis of
$\CH$, and we have the relations (\cite[3.3(b), Lemma 3.4,
Propositions 3.6,3.7]{L1}):
\begin{equation}
\theta_x\cdot \theta_{x'}=\theta_{x+x'},\text{ for all }x,x'\in X;\ 
  \theta_0=1;
\end{equation}
\begin{equation}\theta_x T_s-T_s\theta_{s(x)}=\begin{cases}
    (\bbq(s)^2-1)\frac{\theta_x-\theta_{s(x)}}{1-\theta_{-\al}},&\text{if
    }\al^\vee\notin 2X^\vee,\\
((\bbq(s)^2-1)+\theta_{-\al}(\bbq(s)\bbq(\wti s)-\bbq(s)\bbq(\wti
s)^{-1}))\frac{\theta_x-\theta_{s(x)}}{1-\theta_{-2\al}}, &\text{if
}\al^\vee\in 2X^\vee,
\end{cases}
\end{equation}
$s=s_\al\in S,\ x\in X$. Here $\wti s$ is defined as follows. Let $S(\al)\subset S^a$ be the
connected component of the Coxeter graph containing $s$. When
$\al^\vee\in 2X^\vee,$ $S(\al)$ must be of affine type $\wti C_l$,
$l\ge 1$. Let $\wti s$ be the image of $s$ under the nontrivial graph
automorphism of $\wti C_l.$

\smallskip

Denote $\CA=\Lambda[\theta_x: x\in X]$, an abelian subalgebra of $\C
H$. The center of $\C H$ is $\C Z=\CA^W$ (\cite[Corollary 3.8]{L1}). In
particular, $\CH$ is finite over its center and the simple
$\CH$-modules are finite dimensional. Let $R(\CH)$ be the Grothendieck
group of finite dimensional $\CH$-modules.

The central characters are
identified with elements of $W\backslash T\times\text{Spec}\Lambda,$
where $T=\Hom_\bZ[X,\bC^\times].$

\subsection{}\label{s:induction} Let $J\subset \Pi$ be given, and put $J^\vee=\{\al^\vee:
  \al\in J\}. $ Let $R_J=R\cap \bQ J$ and
$R_J^\vee=\{\al^\vee: \al\in R_J\}.$ Let $W_J\subset W$ be the
parabolic subgroup defined by the reflections in $J$. Denote $\wti
W_J=X\rtimes W_J,$ and let $\bbq_J$ be the restriction of $\bbq$ to
$\wti W_J.$ Consider the root datum $\Phi_J=(X,R_J,X^\vee,R_J^\vee,J)$
and the affine Hecke algebra $\CH_J=\CH(\Phi_J,\bbq_J).$ This algebra
can be identified with the subalgebra of $\C H(\Phi,\bbq)$ generated
by $T_w,$ $w\in W_J$ and $\theta_x,$ $x\in X.$ One calls it a
parabolic subalgebra of $\C H$.

Define the induction functor 
\begin{equation}\label{e:induction}
i_J:R(\C H_J)\to R(\CH).
\end{equation}

To define the ``semisimple part'' of $\CH$, one introduces
\begin{equation}
X_J=X/X\cap (J^\vee)^\perp\text{ and } X_J^\vee=X^\vee\cap \bQ J^\vee,
\end{equation}
where $(J^\vee)^\perp=\{x\in X: \langle x,\al^\vee\rangle=0,\text{ for
  all }\al\in J\}.$ Consider the root datum
$\Phi_J^{\sem}=(X_J,R_J,X_J^\vee,R_J^\vee, J)$ and the affine Hecke
algebra $\CH_J^{\sem}=\CH(\Phi_J^{\sem},\bbq_J).$ For every 
\begin{equation}
t\in T^J=\Hom_\bZ(X/X\cap \bQ J,\bC^\times),
\end{equation}
let $\chi_t:\CH_J\to \CH_J^{\sem}$ be the algebra homomorphism (cf. \cite{OS1})
\begin{equation}\label{chi-t}
\chi_t(\theta_x T_w)=t(x) \theta_{x_J} T_w,\ x\in X, w\in W,
\end{equation}
where $x_J$ is the image of $x$ in $X_J.$ For every $\sigma\in
R(\CH_J^{\sem})$ and every $t\in T^J,$ one can therefore construct the
parabolically induced module
\begin{equation}
X(J,\sigma,t)=i_J(\sigma\circ \chi_t).
\end{equation}

\section{Clifford Theory for $\CH \rtimes \langle\delta\rangle$}\label{sec:clifford}

In this section, we consider the affine Hecke algebras $\CH$ together with an automorphism $\d$ and the Clifford theory. Notice that even if one is just interested in the representation theory of $\CH$, automorphisms of its parabolic subalgebras appear in the study of induction and restriction functors. See also \S \ref{sec:3}.

\subsection{} Suppose $\delta$ is an automorphism of $\Phi$ of finite
order $d$. If the indeterminates $\bbq$ satisfy
$\bbq(\delta(w))=\bbq(w),$ for all $w\in \wti W$, we can define an
extension of the affine Hecke algebra $\CH$ by $\Gamma=\langle\delta\rangle$:
\begin{equation}
\CH'=\CH\rtimes\Gamma.
\end{equation}
Set $\wti W'=\wti W \rtimes \Gamma$ and $W'=W\rtimes \Gamma.$ The center of $\CH'$ is
$\CA^{W'}$, so the central characters are parameterized by points in
$W'\backslash T\times\text{Spec}\Lambda.$

If $J\subset \Pi,$ set $\Gamma_J=\{\delta^i\in\Gamma:
\delta^i(R_J)=R_J\}.$ The parabolic subalgebra is then
$\CH_J'=\CH_J\rtimes\Gamma_J.$ We denote the induction functor again
by $i_J:R(\CH_J')\to R(\CH').$

\subsection{}\label{s:d-comm} We define the $\delta$-commutators and cocenters. This
section is analogous with \cite[section 3.1]{CH}, where the similar
notions in the setting of the graded affine Hecke algebra were considered.

\begin{definition}
If $h,h'\in \CH$, define the $\delta^i$-commutator of $h$ and $h'$ by
$[h,h']_{\delta^i}=hh'-h'\delta(h).$ Let $[\CH,\CH]_{\delta^i}$ be the
submodule of $\CH$ generated by all $\delta^i$-commutators.
\end{definition}
Denote by $\bar\CH^{[i]}$ the quotient of $\CH/[\CH,\CH]_{\delta^i}$
by the image of $(1-\delta)$. The following result was proved in
\cite{CH} for extended graded Hecke algebras, but the proof applies to any
associative algebra extended by a finite cyclic group.

\begin{proposition}[{\cite[Proposition 3.1.1]{CH}}]\label{cocenter-decomp} Set
  $\bar\CH'=\CH'/[\CH',\CH'].$ Then:
\begin{enumerate}
\item $\bar\CH'=\oplus_{i=0}^{d-1}
  \CH\delta^i/([\CH',\CH']\cap\CH\delta^i)$;
\item The map $h\mapsto h\delta^i$ induces a linear isomorphism from
  $\bar\CH^{[i]}\to \CH\delta^i/([\CH',\CH']\cap\CH\delta^i)$.
\end{enumerate}
\end{proposition}

\subsection{}\label{sec:2.7} We discuss Clifford theory for $\CH'$. This is standard and analogous with the graded affine Hecke algebra case from \cite[section 3.2]{CH}, and the proofs are identical.

Let $\Gamma=\<\d\>$. If $(\pi,X)$ is a finite dimensional $\CH$-module, let $(^{\delta^i} \pi,{}^{\delta^i}\!X)$
denote the $\CH$-module with the action
$^{\delta^i}\!\pi(h)x=\pi(\delta^{-i}(h))x,$ for all $x\in X,$ $h\in
\CH.$ If $X$ is a simple module, define the inertia group
$\Gamma_X=\{\delta^i: X\cong{}^{\delta^i}\!X\}.$ 

Fix a family of isomorphisms $\phi_{\delta^i}: X\to {}^{\delta^{-i}}\!
X$, $\delta^i\in \Gamma_X$  such that $\phi_{\delta^{k
    i}}=\phi_{\delta^{i}}^k.$ This is possible since $\Gamma_X$ is cyclic.

If $U$ is an irreducible $\Gamma_X$-module, there is an action of
$\CH\rtimes \Gamma_X$ on $X\otimes U$: 
\begin{equation}
(h\delta^i)(x\otimes u)=h \phi_{\delta^i}(x)\otimes \delta^i u.
\end{equation}
One can form the induced $\CH'$-module $X\rtimes U=\Ind_{\CH\rtimes
  \Gamma_X}^{\CH\rtimes\Gamma} (X\otimes U).$ 

\begin{theorem}[{cf. \cite[Appendix A]{RR}}]\label{t:irrH'} \begin{enumerate}
\item If $X$ is an irreducible $\CH$-module and $U$ an irreducible
  $\Gamma_X$-module, the induced $\CH'$-module $X\rtimes U$ is
  irreducible.
\item Every irreducible $\CH'$-module is isomorphic to an $X\rtimes U.$
\item If $X\rtimes U\cong X'\rtimes U'$, then $X,X'$ are
  $\Gamma$-conjugate $\CH$-modules, and $U\cong U'$ as $\Gamma_X$-modules.
\end{enumerate}
\end{theorem}

 For every
$\delta'\in\Gamma$, set $\Irr^{\delta'}\CH=\{X\in\Irr\CH: \delta'\in\Gamma_X\}$ and let $R^{\delta'}(\CH)$ denote the $\bZ$-linear span of $\Irr^{\delta'}\CH.$ If $(\pi,X)\in\Irr^{\delta'}\CH$, let
$\phi_{\delta'}\in \End_\bC(X)$ be the intertwiner
as before. The twisted trace is 
$$\tr^{\delta'}(\pi):\CH\to\bC,\
\tr^{\delta'}(\pi)(h)=\tr(\pi(h)\circ\phi_{\delta'}).$$ 
Let also $\tr(~,~): \CH'\times R(\CH')\to \bC$ be the trace pairing,
i.e., $\tr(h,\pi)=\tr\pi(h),$ $h\in\CH'$, $\pi\in R(\CH').$

\begin{lemma}[{cf. \cite[Lemma 3.2.1]{CH}}]\label{l:traceH'} Let $X\rtimes U$ be an irreducible
  $\CH'$-module as in Theorem \ref{t:irrH'}. For $h\in \CH,$ $\delta'\in\Gamma,$
\begin{equation}
\tr(h\delta',X\rtimes U)=\begin{cases}
  \delta'(U)\sum_{\gamma\in\Gamma/\Gamma_X}\tr^{\delta'}(X)(\gamma^{-1}(h)),
  &\text{ if }\delta'\in\Gamma_X,\\
0, &\text{ if }\delta'\notin\Gamma_X,
\end{cases}
\end{equation}
where $\delta'(U)$ is the root of unity by which $\delta'$ acts in $U$.
\end{lemma}

{Set $R(\CH')_{\mathbb C}=R(\CH') \otimes_{\mathbb Z} \mathbb C$. 
{Let $\co$ be a $\Gamma$-orbit on $\Irr \CH$. Set $\Gamma_\co=\Gamma_X$ for any $X \in \co$. This is well-defined since $\Gamma$ is cyclic. Then for any irreducible $\Gamma_\co$-module $U$ and $X \in \co$, $X \rtimes U=\oplus_{Y \in \co} Y \otimes U$ is independent of the choice of $X$. We denote it by $\co \rtimes U$. By Theorem \ref{t:irrH'}, $\Irr \CH'=\{\co \rtimes U\}$, where $\co$ runs over $\Gamma$-orbits on $\Irr\CH$ and $U$ runs over isomorphism classes of irreducible representations of $\Gamma_\co$. 

Suppose that $\d^i \in \Gamma_\co$. Let $U_{\co, i}$ be the virtual representation of $\Gamma_\co$ whose character is the characteristic function on $\d^i$. Then $\{\co \rtimes U_{\co, i}\}$ is a basis of $R(\CH')_\bC$. 

Let $R^{[i]}(\CH')_\bC$ be the subspace of $R(\CH')_\bC$ spanned by $\co \rtimes U_{\co, i}$, where $\co$ runs over $\Gamma$-orbits on $\Irr\CH$ with $\d^i \in \Gamma_\co$. Then 
\begin{equation}\label{groth-decomp}
R(\CH')_\bC=\bigoplus_{i=0}^{d-1} R^{[i]}(\CH')_\bC.
\end{equation}
By definition, $R^{[i]}(\CH')_\bC$ is a vector space with basis $(\Irr^{\d^i} \CH)_\Gamma$. The map $X \in \Irr^{\d^i}(\CH) \mapsto X \rtimes U_{\Gamma_X, i}$ induces an isomorphism $R^{\delta^i}(\CH)_{\Gamma, \mathbb C} \to R^{[i]}(\CH')_\bC$.  Here $R^{\delta^i}(\CH)_{\Gamma, \mathbb C}$ is the $\Gamma$-coinvariants of $R^{\delta^i}(\CH)_{\mathbb C}$.

By Lemma \ref{l:traceH'}, for $0 \le i, j<d$ with $i \neq j$, $\tr(\CH \d^i, \co \rtimes U_{\co, j})=0$. }

\subsection{} Define the trace linear map
\begin{equation}
\tr: \bar\CH'\to R(\CH')^*,\quad h\mapsto (f_h: R(\CH')\to \bC,\
f_h(\pi)=\tr\pi(h)).
\end{equation}
This map is compatible with the decompositions from Proposition
\ref{cocenter-decomp} and (\ref{groth-decomp}) as follows.
 Let
$R^*_{\delta}(\CH)=\Hom_\bC(R^{\delta}(\CH)_\bC,\bC)$.  
The twisted trace map 
\begin{equation}
\tr_{\delta}: \CH\to R^*_{\delta}(\CH), \quad h\mapsto (f^{\delta}_h:
R^\delta(\CH)_\bC\to\bC, \ f^\delta_h(\pi)=\tr^\delta(\pi))
\end{equation}
descends to a linear map
\begin{equation}
\tr^\delta:\bar\CH_\delta=\CH/[\CH,\CH]_\delta\to R^*_{\delta}(\CH).
\end{equation}

\begin{definition}\label{d:goodforms} A form $f\in R^*_\delta(\CH)$ is called a trace form if $f\in \im \tr^\delta$.
Denote the
subspace of trace forms by $R^*_\delta(\CH)_{\tr}.$

A form $f\in R^*_\delta(\CH)$ is called good if for every $J\subset I$
such that $\delta(J)=J$, and every $\sigma\in\Irr^\delta(\CH_J^\sem)$,
the function $t\mapsto f(X(J,\sigma,t))$ is a regular function on $(T^J)^\delta$. Denote the subspace of good
forms by $R^*_\delta(\CH)_\good.$
\end{definition}
It is clear that $R^*_\delta(\CH)_{\tr}\subset R^*_\delta(\CH)_\good.$

\section{Induction and restriction}\label{sec:3}

In this section, we recollect some known facts about the induction and restriction functors. 

\subsection{} If $K\subset
J(\subset \Pi)$ are given, denote by $i_K^J: R(\CH_K)\to R(\CH_J)$ the
functor of induction, and by $r_K^J: R(\CH_J)\to R(\CH_K)$ the functor
of restriction. When $J$ and $K$ are $\delta$-invariant, we also have the corresponding functors, denoted again
by $i_K^J$ and $r_K^J$ between $R^\delta(\CH_K)$ and $R^\delta(\CH_J).$

It is obvious that 

\begin{lemma}\label{l:ind-restr-1}
For $L\subset K\subset J\subset I,$ $i_L^J=i_K^J\circ i_L^K$ and $r_L^J=r_L^K\circ r_K^J.$
\end{lemma}

The following result. is the analogue of \cite[Lemma 5.4]{BDK} and
\cite[Lemma 2.1]{Fl}. See also \cite[Theorem 1]{Mi}.

\begin{lemma}\label{l:ind-restr-2}
If $K=\delta(K)$ and $J=\delta(J)$, then
\begin{equation}
r_K\circ i_J=\sum_{w\in {}^K W^J \cap W^{\delta}} i_{K_w}^K\circ w\circ r_{J_w}^J,
\end{equation}
as functors from $R^\delta(\CH_J)$ to $R^\delta(\CH_K)$, where ${}^K W^J$ is the set of representatives of minimal length for $W_K \backslash W/W_J$, $K_w=K\cap wJw^{-1}$ and $J_w=J\cap w^{-1}K w.$
\end{lemma}

\subsection{}
We recall certain elements defined in \cite[section 5.1]{L1}. Let $\C F$ be the quotient field of $\C Z$. For every $\al\in \Pi$, set
\begin{equation}
\C G(\al)=\begin{cases}
    \frac{\bbq(s_\al)^2\theta_\al-1}{\theta_{\al}-1},&\text{if
    }\al^\vee\notin 2X^\vee,\\
\frac{(\bbq(s)\bbq(\wti s)\theta_\al-1)(\bbq(s)\bbq(\wti
s)^{-1})\theta_\al+1)}{\theta_{2\al}-1}, &\text{if
}\al^\vee\in 2X^\vee,
\end{cases}
\end{equation}
see \cite[section 3.8]{L1}. We have $\C G(\al)\in \C A_\C F=\C A\otimes_{\C Z}\C F.$ Define the following elements in $\C H_\C F=\C H\otimes_{\C Z}\C F$:
\begin{equation}
\tau^\al=(T_{s_\al}+1)\C G(\al)^{-1}-1.
\end{equation}
By \cite[Proposition 5.2]{L1}, the assignment $s_\al\mapsto \tau^\al$, $\al\in \Pi$ extends to a unique group homomorphism $W\to \CH_{\C F}^\times.$ Denote by $\tau_w$ the image of $w\in W.$ Moreover,
\begin{equation}\label{commute-tau-A}
f\tau_w=\tau_w w^{-1}(f), \quad \text{ for all } w\in W,\ f\in \C A_\C F.
\end{equation}
The following lemma is also well-known, see for example \cite[section 1.6]{BM3}, where a similar statement was verified in the context of graded Hecke algebras.

\begin{lemma}\label{l:commute-tau-J}
Suppose $K,J\subset \Pi$ are such that $K=w(J),$ where $w\in {}^KW^J$. For every $\al\in J,$ $T_{s_\al}\tau_{w^{-1}}=\tau_{w^{-1}}T_{s_\beta}$, where $\beta=w(\al)\in K.$
\end{lemma}

\begin{proof}
We calculate:
\begin{align*}
T_{s_\al}\tau_{w^{-1}}&=((\tau^\al+1)\C G(\al)-1)\tau_{w^{-1}}\\
&=(\tau^\al+1)\tau_{w^{-1}} \C G(\beta)-\tau_{w^{-1}},\quad \text{by (\ref{commute-tau-A})}\\
&=(\tau_{w^{-1}}\tau^\beta+\tau_{w^{-1}})\C G(\beta)-\tau_{w^{-1}},\quad \text{since $\tau$ is a homomorphism}\\
&=\tau_{w^{-1}}T_{s_\al}.
\end{align*}
\end{proof}

\begin{lemma}\label{l:w-twist}
 If $\delta(K)=K$, $\delta(J)=J$, $w\in {}^KW^J(\delta)$, and $K=w(J),$ then $i_K\circ w=i_J$ in $R^\delta(\CH).$
\end{lemma}

\begin{proof}
Let $(\sigma\circ\chi_t,V)$ be an irreducible module in $R^\delta(\C H_J),$ $t\in T^J$, and let $(\sigma^w\circ \chi_{w(t)},V^w=V)\in R^\delta(H_K)$ be its twist by $w.$ We need to prove that in $R^\delta(\CH),$ $i_K(\sigma^w\circ\chi_{w(t)})= i_J(\sigma\circ\chi_t).$ Since the characters of both sides are regular functions in $t$, it is sufficient to prove this for $t$ generic. Define 
\begin{equation}
\phi_w(\sigma,t): \C H\otimes_{\C H_J}V \to \CH\otimes_{\C H_K} V^w,\ h\otimes v\mapsto v \tau_{w^{-1}}\otimes v.
\end{equation} 
This is a well-defined intertwining operator. To see this, notice
first that since we are assuming $t$ is generic, $\tau_{w^{-1}}$
evaluated at $\sigma\otimes\chi_t$ has no poles. Secondly, by
(\ref{commute-tau-A}) and Lemma \ref{l:commute-tau-J}, if $h'\in \C
H_J,$ then $h'\tau_{w^{-1}}=\tau_{w^{-1}}w(h')$, and thus
$\phi_w(\sigma,t)(hh'\otimes v)=\phi_w(\sigma,t)(h\otimes (\sigma(h')\circ\chi_t) v).$ Since the action of $\C H$ on the induced modules is by left multiplication, $\phi_w$ is indeed an intertwiner. Finally, for generic $t$, $\phi_w(\sigma,t)$ is invertible, and the inverse is $\phi_{w^{-1}}(\sigma^w,w(t)).$
\end{proof}

\subsection{} For every $J\subset \Pi$, let $\wti i_J: \CH'_J\to \CH'$
denote the inclusion. Define $\wti r_J:\CH'\to \CH'_J$ as follows. Given
$h\in \CH'$, let $\psi_h:\CH'\to\CH'$ be the linear map given by left
multiplication by $h.$ This can be viewed as a right $\CH'_J$-module
morphism. Since $\CH'$ is free of finite rank right $\CH'_J$-module, one can consider $\tr\psi_h\in\CH'_J.$ 
Set $\wti
r_J(h)=\tr\psi_h.$ 

As before, for every $K\subset J,$ we may also
define $\wti i_K^J$ and $\wti r_K^J.$

Suppose $\delta(J)=J.$ Define $\bar \CH_{J,\delta}=\CH_J/[\CH_J,\CH_J]_\delta$. Then $\wti i_J$ gives rise to a well-defined map (not necessarily injective)
$$\bar i_J: \bar\CH_{J,\delta}\to \bar\CH_\delta.$$
Since $\wti r_J[\CH,\CH]_\delta\subset [\CH_J,\CH_J]_\delta$, $\wti r_J$ descends to a well-defined map
$$\bar r_J:\bar\CH_\delta\to \bar\CH_{J,\delta}$$ sending $h+[\CH, \CH]_\d$ to the image of $\wti r_J(h \d) \d \i$ in $\CH_{J, \d}$. 
Define the similar notions $\bar i_K^J$ and $\bar r_K^J.$ 

\begin{lemma}[{cf. \cite[Lemma 4.5.1]{CH}}]\label{l:adjointtrace}
The maps $\bar i_J$ and $\bar r_J$ are $\tr(~,~)$-adjoint to $r_J$ and
$i_J$, respectively.
\end{lemma}

\subsection{} 
We introduce the following hypothesis:

\medskip

\noindent (PDT) The Density Theorem holds for every proper parabolic subalgebra $\CH_J$, $\delta(J)=J$.

\medskip

The following properties are dual to those of $i_J$ and $r_J.$

\begin{lemma}\label{l:ind-restr-bar} Suppose (PDT) holds. Then 
\begin{enumerate} 
\item[(i)] For $L\subset K\subset J$, $\bar i_L^J=\bar i_K^J\circ \bar i_L^K$ and $\bar r_L^J=\bar r_L^K\circ \bar r_K^J.$
\item[(ii)] If $K=\delta(K)$ and $J=\delta(J),$ then
\begin{equation}
\bar r_J\circ \bar i_K=\sum_{w\in {}^K W^J(\delta)} \bar i_{J_w}^J\circ w^{-1}\circ \bar r_{K_w}^K,
\end{equation}
as maps from $\bar \CH_{K,\delta}\to \bar\CH_{J,\delta}$, where $K_w=K\cap wJw^{-1}$ and $J_w=J\cap w^{-1}K w.$
\item[(iii)] If $w\in {}^KW^J(\delta)$ and $K=w(J)$, 
\begin{equation}
w^{-1}\circ \bar r_K=\bar r_J.
\end{equation} 
\end{enumerate}
\end{lemma} 

\begin{remark}
(1) We will see in section \ref{s:8} that the (PDT) hypothesis is not necessary. 

(2) In general, if $J$, $K$, and $w$ are as in Lemma
\ref{l:ind-restr-bar}(iii), $\bar i_K\circ w\neq \bar i_J$ in $\bar
\CH_\delta.$ For example, suppose $\delta=1,$ $J=K=\emptyset$, and take
$\theta_x\in \bar \CH_{\emptyset}.$ Then, for $w\in
W,$ $\theta_x\not\equiv\theta_{w(x)}$ in $\bar \CH$, in general. To see
this, one can use the (one-dimensional) Steinberg module $\St$ of
$\CH$, and the fact that $\tr(\theta_x,\St)\neq
\tr(\theta_{w(x)},\St)$, in general.
\end{remark}

\subsection{} In the end of this section, we recall the $A$-operators.

As in \cite[section 5.5]{BDK}, see also \cite[section 2]{Da}, fix an order of the subsets $K$ of every given size, and define 
\begin{equation}\label{A-map}
A_\ell=\prod_{K=\delta(K), |K|=|\Pi|-\ell}(i_K \circ r_K-|N_K|),\text{ and }
A=A_{|\Pi|}\circ\dots\circ A_1.
\end{equation}
The map $A$ induces a linear map $A_\bC: R^\delta(\CH)_\bC\to R^\delta(\CH)_\bC$ as well.

Let 
\begin{equation}\label{Abar-map}
\bar A^\ell=\prod_{\delta(K)=K,|K|=|\Pi|-\ell} (\bar i_K \circ \bar r_K-|N_K|),\quad \bar A=\bar A^{1}\circ \bar A^{2}\circ\dots \circ \bar A^{|\Pi|},
\end{equation}
be the adjoint operator.

\smallskip

The following proposition is proved in \cite[sections 5.4, 5.5]{BDK} and \cite[Proposition 2.5 (i)]{Da}.

\begin{proposition}\label{p:A} We have $A^2= a A$ for some $a \neq 0$ and $$\ker A_\bC=\sum_{J=\d(J) \subsetneqq I} i_J(R^\d(\CH_J)_\bC).$$ 
\end{proposition}

The following proposition is the dual result to Proposition \ref{p:A}.

\begin{proposition}\label{p:A-bar} Suppose (PDT) holds. Then 

(1) $\bar A^2=a \bar A$ for some $a \neq 0$. 

(2) Set $\bar \CH^{\EL}_\d=\cap_{J=\d(J) \subsetneqq I} \ker \bar r_J$. Then $\bar\CH^{\EL}_\delta=\im\bar A$.
\end{proposition}

\section{Spanning set of the cocenter}\label{sec:span-cocenter}

In this section, we recall the explicit description of the cocenter of $\CH$ obtained in \cite{HN} and \cite{HN2}. 

\subsection{} We first recall some properties on the minimal length elements in $\wti W'$. 

We follow \cite{HN}. For $w, w' \in \wti W'$ and $s \in S^a$, we write $w \xrightarrow{s} w'$ if $w'=s w s$ and $\ell(w') \le \ell(w)$.  We write $w \to w'$ if there is a sequence $w=w_0, w_1, \cdots, w_n=w'$ of elements in $\wti W$ such that for any $k$, $w_{k-1} \xrightarrow{s} w_k$ for some $s \in S^a$.


We call $w, w' \in \wti W'$ {\it elementarily strongly conjugate} if $\ell(w)=\ell(w')$ and there exists $x \in \wti W'$ such that $w'=x w x \i$ and $\ell(x w)=\ell(x)+\ell(w)$ or $\ell(w x \i)=\ell(x)+\ell(w)$. We call $w, w'$ {\it strongly conjugate} if there is a sequence $w=w_0, w_1, \cdots, w_n=w'$ such that for each $i$, $w_{i-1}$ is elementarily strongly conjugate to $w_i$ and we write $w \sim w'$ in this case. 

It is proved in \cite[Theorem A]{HN} that

\begin{theorem}\label{min}
Let $\co$ be a conjugacy class of $\wti W'$ and $\co_{\min}$ be the set of minimal length elements in $\co$. Then

(1) For any $w' \in \co$, there exists $w'' \in \co_{\min}$ such that $w' \rightarrow w''$.

(2) Let $w', w'' \in \co_{\min}$, then $w'  \sim w''$.
\end{theorem}

\subsection{} By definition, if $w \sim w'$, then $T_w \equiv T_{w'} \mod [\CH', \CH']$. Let $\co$ be a conjugacy class of $\wti W'$. Let $T_\co$ be the image of $T_w$ in $\bar \CH'$, where $w$ is a minimal length element in $\co$. By Theorem \ref{min} (2), $T_\co$ is independent of the choice of $w$. 

Moreover, we have the following result.

\begin{theorem}\label{IMbasis}
The elements $\{T_\co\}$, where $\co$ ranges over all the conjugacy classes of $\wti W'$, span $\bar \CH'$ as an $\Lambda$-module.
\end{theorem}

The equal parameter case was proved in \cite[Theorem 5.3]{HN}. The general case can be proved in the same way and we omit the details. In \cite[Theorem 6.7]{HN}, it is proved the in the equal parameter case, $\{T_\co\}$ is a basis of $\bar \CH'$. The proof there is to use base ring to reduce the question to the group algebra of $\wti W'$ and then to prove the density theorem for the group algebra. 

In this paper, we will prove that $\{T_\co\}$ is a basis of $\bar \CH'$ for arbitrary after establishing the trace Paley-Wiener Theorem and density Theorem. 

\subsection{} The elements $\{T_\co\}$ are called the Iwahori-Matsumoto presentation of $\bar \CH'$. However, the expression for a minimal length element $T_w$ in $\co$ is usually very complicated. To study the induction and restriction functors on the cocenter, we also need the Bernstein-Lusztig presentation of $\bar \CH'$ established in \cite{HN2}. 


\subsection{} 
Let $X_\bC=X \otimes_\bZ \bC$ and $N=|W'|$. For any $w \in \wti W'$,
$w^n=t_{\lambda}$ for some $\lambda \in X$. Let $\nu_w=\lambda/N$ and
$\bar \nu_{w}$ be the unique dominant element in the $W_0$-orbit of
$\nu_w$. Then the map $\wti W \to X_\bC, w \mapsto \bar \nu_{w}$ is
constant on the conjugacy class of $\wti W'$. For any conjugacy class
$\co$, we set $\nu_\co=\bar \nu_{w}$ for any $w \in \co$ and call it
the {\it Newton point of} $\co$. We set 
\begin{equation}\label{e:JO}
J_\co=\{\al \in \Pi: s_\al(\nu_\co)=\nu_\co\}.\end{equation} 

\subsection{} Let $p: \wti W' \to W'$ be the projection map. We call an element $w \in W'$ {\it elliptic} if $(X_\bC)^w \subset (X_\bC)^{W}$ and an element $w \in \wti W$ {\it elliptic} if $p(w)$ is elliptic in $W'$. A conjugacy class $\co$ in $\wti W'$ is called {\it elliptic} if $w$ is elliptic for some (or, equivalently any) $w \in \co$. By definition, if $\co$ is elliptic, then $J_\co=\Pi$. 

\subsection{}\label{3.6} Let $\mathcal A$ be the set of pairs $(J, C)$, where $J \subset S$, $C$ is an elliptic conjugacy class of $\wti W_J \rtimes \Gamma_J$ and $\nu_{w}$ is dominant for some (or, equivalently any) $w \in C$. For any $(J, C), (J', C') \in \mathcal A$, we write $(J, C) \sim (J', C')$ if $\nu_{w}=\nu_{w'}$ for $w \in C$ and $w' \in C'$ and there exists $x \in {}^{J'} (W_{J_{\nu_{w}}} \rtimes \Gamma_{J_{\nu_{w}}})^J$ such that $x J x \i=J'$ and $x C x \i=C'$. 

By \cite[Lemma 5.8]{HN2}, the map from $\mathcal A$ to the set of conjugacy classes of $\wti W'$ sending $(J, C)$ to the unique conjugacy class $\co$ of $\wti W$ with $C \subset \co$ gives a bijection from $\mathcal A/\sim$ to the set of conjugacy classes of $\wti W$. We denote this map by $\tau$. It is proved in \cite[Theorem B]{HN2} that

\begin{theorem}\label{BLbasis}
Let $(J, C) \in \mathcal A$ and $\co=\tau(J, C)$. Then $$T_{\co}=\bar i_J(T^J_C).$$
\end{theorem}

\begin{remark}
Here $T^J_C$ is the image of $T_y$ in $\bar \CH'_J$ for any minimal length element $y$ of $C$ (with respect to the length function on $\wti W_J \rtimes \Gamma_J$). However, in general, $y$ is not a minimal length element in $\co$ (with respect to the length function on $\wti W'$). The expression of $\bar i_J(T^J_C)$ involves the Bernstein-Lusztig presentation of $\CH'_J$ and $\CH'$. 
\end{remark}

\subsection{}\label{BLbasis1} As we'll see later, the conjugacy classes with central Newton points play a crucial role in this paper. Here we discuss a variation of \S\ref{3.6} and Theorem \ref{BLbasis}. 

Let $\co$ be a conjugacy class of $\wti W'$ and $J=J_\co$. We may assume that $\co=\tau(K, C)$. Then it is easy to see that $K \subset J$. Let $\co'$ be the conjugacy class of $\wti W'_J \rtimes \Gamma_J$ such that $\co'=\tau_J(K, C)$. Then $J_{\co'}=J$ and $T^J_{\co'}=\bar i_K^J(T^K_C)$. Hence $$T_\co=\bar i_K(T^K_C)=\bar i_J (T^J_{\co'}).$$

\section{Elliptic quotient and rigid quotient}\label{s:6}

In this section, we discuss two natural quotients for $R^\d(\CH)$: the elliptic quotient introduced in \cite{BDK} and the rigid quotient, which we introduce below.

\subsection{}

The elliptic quotient of $R^\d(\CH)$ is defined to be $$\bar R^\d_0(\CH)_\bC=R^\d(\CH)_\bC/\sum_{J=\d(J) \subsetneqq I} i_J (R^\d(\CH_J)_\bC).$$ 

Now we introduce the rigid quotient. 

\begin{definition}\label{d:Hecke-rigid}
Set
\begin{equation}
\begin{aligned}
R^\delta(\CH)_{\twist}=\text{span}\{i_J(\sigma)-i_J(\sigma\circ\chi_t):&~
J=\delta(J),~\sigma\in
R^\delta(\CH^\sem_J),~\\& t\in\Hom_\bZ(X\cap \bQ R/X\cap \bQ
J,\bC^*)^\delta\}\\
\end{aligned}
\end{equation}
and define the rigid quotient of $R^\d(\CH)$ to be
\begin{equation}
\bar R^\delta(\CH)_{\rig}:=R^\delta(\CH)/R^\delta(\CH)_{\twist}.
\end{equation}
\end{definition}
Notice in the definition the presence of $X\cap \bQ R$, rather than
$X$. This is to account for the fact the the root datum may not be semisimple.

\

In the rest of this section, we assume that the root datum $\Phi$ for $\CH$ is semisimple. We estimate the dimension of the elliptic quotient and rigid quotient of $R^\d(\CH)$.

\begin{proposition}\label{p:finite-elliptic}
Suppose (PDT) holds and that  the root datum for $\CH$ is semisimple.
Then the rank of the restriction of the trace map $\tr:\bar\CH^{\EL}_\delta\times \bar R^\delta_0(\CH)\to \bC$ equals $\dim \bar R^\delta_0(\CH).$ In particular, $\dim \bar R^\delta_0(\CH)\le \dim \bar \CH^{\EL}_\delta.$
\end{proposition}

\begin{proof}
Suppose $\{\pi_1,\pi_2,\dots,\pi_k\}$ is a set in $\subset
  R^\delta(\CH)$ such that its image in $\overline R^\delta_0(\CH)$ is
  linearly independent. Applying the operator $A$, one obtains a linearly independent set
  $\{A(\pi_1),A(\pi_2),\dots,A(\pi_k)\}$ in $R^\delta(\CH).$ This is
  because $A(\pi)\equiv a \pi$ in $\overline R^\delta_0(\CH)$, and
  $a\neq 0.$

  Since the characters of simple modules are linear independent, so
  are the characters of any linear independent set in $R^\delta(\CH).$
  Thus there exist elements $h_1,h_2,\dots,h_k$ of $\bar\CH$, such
  that the matrix $(\tr(h_i\delta,A(\pi_j))_{i,j}$ is invertible. By Lemma
  \ref{l:adjointtrace}, the matrix $(\tr(\bar A(h_i)\delta,\pi_j))_{i,j}$ is also
  invertible, hence $\{\bar A(h_i)\}$ is a linear independent set in $\bar\CH_\delta.$ Now the Proposition follows from
  Proposition \ref{p:A-bar}.
\end{proof}

\subsection{} In fact, we will see in section \ref{s:8} that in all the statements in this section, the (PDT) assumption can be dropped. 

\subsection{} Let $\CI^\d=\{J \subset \Pi; J=\d(J)\}$. For $J, J' \in \CI^\d$, we write $J \sim_\d J'$ if there exists $w \in W^\d \cap {}^{J'} W^J$ such that $w(J)=J'$. For each $\sim_\d$-equivalence class in $\CI^\d$, we choose a representatives. We denote by $\CI^\d_\spadesuit \subset \CI^\d$ the set of representatives. Set $N_J=\{z \in {}^J W^J; z=\d(z), J=z(J)\}$. Then $N_J$ acts on $R^\d(\CH^{\sem}_J)$ and on $\bar R^\d_0(\CH^{\sem}_J)$. 

By Lemma \ref{l:ind-restr-1} and Lemma \ref{l:ind-restr-2}, we have 

\begin{lemma}\label{l:restr-rigid}
If $K=\delta(K)\subsetneq\Pi,$ then $i_K(R^\delta(\CH_K)_\twist)\subseteq R^\delta(\CH)_\twist$ and  $r_K(R^\delta(\CH)_\twist)\subseteq R^\delta(\CH_K)_\twist.$
\end{lemma}

The first result is a dimension estimate for $\bar R^\delta(\CH)_\rig.$

\begin{proposition}\label{p:dim-rigid} Suppose that the root datum for $\CH$ is semisimple. Then
$$\dim\bar R^\d(\CH)_{\rig}=\sum_{J \in \CI^\d_\spadesuit} \dim
\bar R^\d_0(\CH_J^\sem)^{N_J}$$
\end{proposition}

\begin{proof}

Let $m=|\Pi|$ and consider first the natural
projection
$$p_m:\bar R^\d(\CH)_\rig\to \bar R^\d_0(\CH).$$
By definition, $$\ker p_m=\sum_{J=\delta(J)\subsetneq\Pi}
[i_J(R^\delta(\CH_J)_\bC)],$$ where $[i_J(R^\delta(\CH_J)_\bC)]
=i_J(R^\delta(\CH_J)) /i_J(R^\delta(\CH_J))\cap R^\d(\CH)_\twist.$ By
Lemma \ref{l:w-twist}, we can replace the right hand side by
$$\ker p_m=\sum_{J \in \CI^\d_\spadesuit}
[i_J(R^\delta(\CH_J)^{N_J}_\bC)].$$
Let $r_J^\sem=\chi_1\circ r_J: R^\d(\CH))_\bC\to R^\d(\CH_J)_\bC\to
R^\d(\CH_J^\sem).$ Let $J_1,\dots, J_k$ be the representatives in
$\CI^\d_\spadesuit$ of the maximal $\d$-stable proper subsets of
$\Pi.$ Then consider:
$$p_{m-1}:\ker p_m\to \bigoplus_{i=1}^k
\bar R_0(\CH_{J_i}^\sem)^{N_{J_i}}_\bC,\ p_{m-1}=\sum_i r_{J_i}^\sem.$$
By Lemma \ref{l:ind-restr-2}, this map is well-defined. It is also
surjective, given the definition of $\ker p_N$.

Continue in this way and define $$p_{m-2}:\ker p_{m-1}\to
\bigoplus_{J\in\CI^\d_\spadesuit, |J|=m-2}
\bar R_0(\CH_{J}^\sem)^{N_{J}}_\bC,\quad p_{m-2}=\sum_{J\in\CI^\d_\spadesuit, |J|=m-2} r_{J}^\sem,$$
etc., until $p_0: \ker p_1\to \bar R_0(\CH_{\emptyset}^\sem)^{W}_\bC=\{0\}.$

Therefore, $\dim \bar R^\d(\CH)_\rig=\sum_{j=0}^N \dim\im p_j=\sum_{J \in \CI^\d_\spadesuit} \dim
\bar R^\d_0(\CH_J^\sem)^{N_J}$.
\end{proof}

\begin{definition}\label{d:generic}
The parameter $\bbq$ of $\CH$ is called admissible if Lusztig's reduction \cite[Theorems 8.6, 9.3]{L1} from affine Hecke algebras to graded affine Hecke algebras holds. For example, this is the case if $\bbq(\al)=q^{L(\al)}$, where $L(\al)\in\bZ_{\ge 0}$ and $q\in\bC^\times$ has infinite order, or more generally, $q$'s order is not small.
\end{definition}

\subsection{} Now we give a precise formula for the dimensions of elliptic quotient and rigid quotient of $R^\d(\CH)$ in terms of $\d$-conjugacy classes of $\wti W$. 

Let $cl(\wti W, \d)$ be the set of $\d$-conjugacy classes of $\wti W$. Let $cl(\wti W, \d)_{\EL}$ be the set of elliptic $\d$-conjugacy classes of $\wti W$, and let $cl(\wti W, \d)_0$ be the set of $\d$-conjugacy classes $\co$ of $\wti W$ such that $\nu_\co=0$. Then $$cl(\wti W, \d)_{\EL} \subset cl(\wti W, \d)_0 \subset cl(\wti W, \d).$$
Let $cl(W, \d)_{\EL} \subset cl(W, \d)$ denote the set of elliptic
$\d$-conjugacy classes, respectively all $\d$-conjugacy classes in
$W$. 

By \cite[Lemma 5.8]{HN2}, we have 
\begin{equation}\label{e:HN2-dims}
\sum_{J \subset \CI^\d_\spadesuit} |cl(\wti W_J,
\d)_{\EL}/N_J|=|cl(\wti W, \d)_0|,
\end{equation}

An element $s\in T=\Hom_\bZ(X,\bC^*)$ is called isolated if there
exists a $\d$-elliptic element $w\in W$ such that $w\cdot s=s.$ By
\cite[Theorem 3.2]{OS2}, Clifford-Mackey induction gives a natural
isomorphism
\begin{equation}
\bar R^\delta_0(\wti W)_\bC\cong \bigoplus_{s\in T/W'\text{ isolated}}
\bar R^\delta(W_s)_\bC.
\end{equation}
Now suppose that we are given a finite group of automorphisms $N$ of
$\wti W$ that preserves $S^a$ and commutes with $\d$. Then the
previous isomorphism implies
\begin{equation}
\bar R^\delta_0(\wti W)_\bC^N\cong \bigoplus_{s\in T/W'N\text{ isolated}}
\bar R^\delta(W_s)_\bC^{N(s)},
\end{equation}
and in particular, considering the dimensions of the spaces involved,
one finds
\begin{equation}\label{6.1.3}
|cl(\wti W,\delta)_{\EL}/N|=\sum_{s} |cl(W_s,\delta)_{\EL}/N(s)|,
\end{equation}
where $s$ ranges over $N$-orbits of isolated representatives in $T/W'$.

We have the following result on the dimension of elliptic quotient. 

\begin{proposition}\label{elliptic}
Suppose that $\bbq$ is admissible and that the root datum for $\CH$ is semisimple. Let $N$ be a group of automorphisms on $\wti W$ that preserves $S^a$ and commutes with $\d$. Then 

(1) $\dim \bar R^\d_0(\CH)_\bC^N=|cl(\wti W, \d)_{\EL}/N|.$

(2) $\dim\bar R^\d(\CH)_\rig^N=|cl(\wti W, \d)_0/N|$.
\end{proposition}

\begin{remark}
In particular, the parameterization of irreducible representations of an affine Hecke algebra with (possibly unequal) admissible parameters is the same as the parameterization of irreducible representations of the extended affine Weyl group once we identify the elliptic quotient for the parabolic subalgebras of affine Hecke algebra and the group algebra. 
\end{remark}

\begin{proof} Let $T_c=\Hom_\bZ(X,S^1)$ be the compact part of the
  torus $T$. For every $s\in T_c/W'$, let $R^\delta(\CH)_s$ denote the span of the irreducible
  $\delta$-stable $\CH$-modules whose central characters have compact
  part $s$. Clearly, we have decompositions
\begin{equation*}
\begin{aligned}
R^\delta(\CH)_\bC&=\bigoplus_{s\in T_c/W'} R^\delta(\CH)_s,\text{ and}\\
R^\delta(\CH)^N_\bC&=\bigoplus_{s\in T_c/W'N} R^\delta(\CH)_s^{N(s)}.
\end{aligned}
\end{equation*}
Let $\bar R^\delta_0(\CH)_s$ denote the image of $R^\delta_0(\CH)$ in
$\bar R^\delta_0(\CH).$ Then, since the action of $N$ preserves
$R_{\Ind}^\delta(\CH)$, we have
\begin{equation*}
\bar R^\delta_0(\CH)^N_\bC=\bigoplus_{s\in T_c/W'N} \bar R^\delta_0(\CH)_s^{N(s)}.
\end{equation*}
If $s$ is not isolated, then $\bar R^\delta_0(\CH)_s=0$, so the
decomposition becomes
\begin{equation}\label{e:sum-isolated}
\bar R^\delta_0(\CH)^N=\bigoplus_{s\in T_c/W'N\text{ isolated}} \bar R^\delta_0(\CH)_s^{N(s)}.
\end{equation}

Let $\bH_s$ be the graded affine Hecke algebra constructed from $\CH$
and the $W$-orbit of $s$ in \cite[section 8]{L1}, see also \cite[section 3]{BM2}. When $s$ is isolated, the root system
corresponding to $\bH_s$ is semisimple. By 
\cite[Theorems 8.3 and 9.2]{L1} and the fact that parabolic induction commutes with the
reduction to graded Hecke algebra (e.g., \cite[Theorem 6.2]{BM2}), it
follows that 
\begin{equation}\label{6.1.4}
\bar R^\delta_0(\CH)_s\cong \bar R^\delta_0(\bH_s)\text{ and }\bar R^\delta_0(\CH)_s^{N(s)}\cong \bar R^\delta_0(\bH_s)^{N(s)}.
\end{equation}
By \cite[Theorems B and C]{CH} applied to the semisimple graded Hecke algebra $\bH_s$, there exists a perfect pairing between the finite
dimensional spaces
\begin{equation}
\tr: \bar R^\delta_0(\bH_s)_\bC\times (\bar\bH_s)_{\delta,0}\to \bC.
\end{equation}
This implies that the subspace $\bar R^\delta_0(\bH_s)_\bC^{N(s)}$ is in perfect duality with the space of
$N(s)$-coinvariants $((\bar\bH_s)_{\delta,0})_{N(s)}.$ By \cite[Theorem 7.2.1]{CH}, $(\bar\bH_s)_{\delta,0}$
has a basis $\{w_C: C\in cl(W_s,\d)_{\EL}\}$, where $w_C$ is a
representative of $C$. This means that
\begin{equation}\label{6.1.6}
\dim \bar R^\delta_0(\bH_s)_\bC^{N(s)}=\dim((\bar\bH_s)_{\delta,0})_{N(s)}=|cl(W_s,\d)_{\EL}/N(s)|.
\end{equation}
Then (\ref{6.1.4}) gives $\dim \bar R^\delta_0(\CH)_s^{N(s)}=|cl(W_s,\d)_{\EL}/N(s)|$, and therefore part (1) of the proposition follows
from (\ref{e:sum-isolated}) and (\ref{6.1.3}).

Part (2) of the proposition follows from part (1), Proposition \ref{p:dim-rigid} and \ref{6.1.3}.
\end{proof}

\subsection{} 
In fact, the proof of Proposition \ref{p:dim-rigid} gives a section of $\bar R^\d(\CH)_\rig\to R^\d(\CH).$ To define it, recall the \cite{BDK}-map defined in (\ref{A-map}) 
$$A: R^\d(\CH)_\bC\to R^\d(\CH)_\bC,$$ 
which by Proposition \ref{p:A} has the property that
\begin{equation}
\ker A=R^\d_1(\CH),\text{ and so }\im A\cong  \bar R^\d_0(\CH).
\end{equation}
Let $A_J:R^\d(\CH_J)_\bC\to R^\d(\CH_J)_\bC$ be the similar maps for the parabolic subalgebras $\CH_J.$ As a consequence of the proof of  Proposition \ref{p:dim-rigid}, we have

\begin{corollary}\label{c:rigid-section} If the root datum for $\CH$ is semisimple, a section of $\bar R^\d(\CH)_\rig$ is given by
$$\bigoplus_{J\in \CI^\d_\spadesuit} i_J\im(A_J:R(\CH_J^\sem)_\bC\to R(\CH_J^\sem)_\bC)^{N_J}.
$$
\end{corollary}

\section{Duality between $\bar\CH$ and $R(\CH)$}\label{s:7}

Our strategy to prove the main results is to use traces of finite dimensional $\CH$-modules to separate the elements $\{T_\co\}$ that span the cocenter $\bar\CH$. 

\subsection{} 
Recall that $R^*_\d(\CH)=\Hom_\bC(R^\d(\CH)_\bC, \bC)$. 
Let $(\bar R^\d(\CH)_\rig)^*$ be the subspace of $R^*_\d(\CH)$ consisting of linear functions on $R^\d(\CH)$ that vanish on $R^\d(\CH)_{\twist}$. Define $$\bar \CH^\rig=\text{span}\{T_\co; J_\co=\Pi\}.$$




We have the following separation theorem.

\begin{theorem}\label{rigid-cocenter}
Suppose  that (PDT)
holds. Then $$\bar\CH_\delta^\rig=\{h \in \bar \CH_\d; \tr_\d(h)(R^\d(\CH)_{\twist})=0\}.$$ 
\end{theorem}

\begin{proof}
Let $\co$ be a conjugacy class with
$J_\co=\Pi$.  There exist $K\subsetneq \Pi^a$ such that
$T_\co=T_w$ (in $\bar \CH_\delta$) for some elliptic element $w$ in
the parahoric subgroup $\wti W'_K$.  If the root datum for $\CH$ is semisimple,  then $\wti W'_K$ is a finite
group, and therefore 
$$i_J(\sigma)\cong i_J(\sigma\circ \chi_t)\text{ as
}\CH_K'\text{-representations},$$
for all $J=\delta(J)\subset\Pi$ and $t\in \Hom(X\cap\bQ R/X\cap \bQ
J,\bC^*)^\delta.$ This means that 
$$\tr(T_w,\pi)=\tr(T_w,\pi),\text{ for all }\pi\in R^\delta(\C H)_{\twist},$$
and hence $T_\co=T_w\in \bar\CH_\delta^\rig.$ The reduction from the reductive root system to the semisimple root system will be discussed in a more general setting in section \ref{s:reduction}.

\smallskip

For the converse inclusion, suppose $\tr_\d(\sum_\co a_\co T_\co)(R^\d(\CH)_{\twist})=0$. Let $J$ be a minimal subset of
$\Pi$ such that $J=J_\co$ for some $\co$ with $a_\co \neq 0$. If $J=\Pi$, we are done. Otherwise, suppose $J\subsetneq \Pi.$ Apply
$i_J(\sigma_{t-t'}):=i_J(\sigma\circ (\chi_t-\chi_{t'}))$, where $\sigma\in R^\delta(\CH_J)$ will be chosen conveniently later, to the
linear combination to obtain:
\begin{align*}
p_J(t,t')&=\sum_\co a_\co \tr(T_\co,
i_J(\sigma,_{t-t'}))\\
&=\sum_{J_\co=J} a_\co\tr(T_\co, i_J(\sigma_{t-t'}))+\sum_{J_\co \neq J} a_\co\tr(T_\co, i_J(\sigma_{t-t'})),
\end{align*}
as a polynomial function in $t,t'\in\Hom(X\cap \bQ R/X\cap \bQ J,\bC^*)^\delta.$ 

In $p_J(t,t')$, the part
$\sum_{J_\co=J}$ is more regular in $t,t'$ than the second part $\sum_{J_\co \neq J}$.
Since by assumption $p_J(t,t')=0$ for all $t,t'$,  it follows that 
$$\sum_{J_\co=J} a_\co \tr(T_\co, i_J(\sigma_{t-t'}))=0.$$ 
By Theorem \ref{BLbasis} and \S\ref{BLbasis1}, for any $\co$ with $J_\co=J$, we have $T_\co=\bar i_J(T^J_{\co'})$ for some conjugacy class $\co'$ of $\wti W_J \rtimes \Gamma_J$ with $J_{\co'}=J$ and $\co' \subset \co$. 

Set $h=\sum_{J_\co=J} a_\co T^J_{\co'}$. By induction, $h\in(\bar\CH_{J,\delta}^\rig)_{N_J}$.Then 
\begin{equation}\label{trace-h}
\sum_{J_\co=J} a_\co \tr(T_\co, i_J(\sigma_{t-t'}))=\sum_{J_\co=J} a_\co
\tr(\bar i_J(T^J_{\co'}), i_J(\sigma_{t-t'}))=\tr(h, r_J\circ i_J(\sigma_{t-t'}))=0.
\end{equation}
By Lemmas \ref{l:ind-restr-2} and \ref{l:w-twist}, we have 
 $$r_J\circ i_J=\sum_{w\in N_J}w+\sum_{w\in
  ^JW^J(\delta)\setminus N_J} i_{J_w}^J\circ r_{J_w}^J,\quad J_w=J\cap w^{-1}Jw.$$
Notice that the second part of the sum involves only $J_w$ that are proper
subsets of $J$.  By Lemma \ref{l:restr-rigid}, we find that $\tr(h,\ )$ vanishes on this part, and so $\tr(h,\sum_{w\in N_J}w\circ \sigma_{t-t'}))=0.$ Since $t,t'$ are arbitrary, this implies that $\sum_{w\in N_J}\tr(h,w\circ\sigma)=0$, for all $\sigma\in R^\delta(\CH_J).$  Now specialize $\sigma\in R^\delta(\CH_J)^{N_J}$ and get 
\begin{equation}\label{e:pair-NJ}
\tr(h,\sigma)=0,\quad \text{for all }\sigma\in R^\delta(\CH_J)^{N_J}.
\end{equation}

The (PDT) assumption says that $\tr:\bar\CH_{J,\delta}\times R^\delta(\CH_J)\to \bC$ is nondegenerate on the left. Passing to $N_J$-coinvariants and $N_J$-invariants, respectively, it follows that $\tr:(\bar\CH_{J,\delta})_{N_J}\times R^\delta(\CH_J)^{N_J}\to \bC$ is also nondegenerate on the left. Then (\ref{e:pair-NJ}) gives $h=0$ in $(\bar\CH_{J,\delta}^\rig)_{N_J}$. Therefore $\sum_{J_\CO=J}a_\CO T_\CO=0$ and this is a contradiction. 

\end{proof}

\subsection{}\label{s:rank} When the parameters of the Hecke algebra are not roots of unity, the rigid cocenter $\bar\CH^\rig$ appears naturally in relation with $K_0(\CH)$, the Grothendieck group of finitely generated projective $\CH$-modules. Let 
\begin{align}
&\Rk: K_0(\CH)\to \bar\CH
\end{align}
be the Hattori-Stallings rank map, see for example \cite[\S1.2]{Da}. If $P\in K_0(\CH)$ is a finitely generated projective module, then $P$ is a direct summand of $\CH^n$ for some positive integer $n$. As such, there exists an idempotent $e_P\in M_{n\times n}(\CH)$ such that $P=\CH^n e_P$ as left $\CH$-modules. The rank map is defined to be the image in $\bar\CH$ of the matrix trace of $e_P$, i.e., 
\begin{equation}
\Rk(P)=\operatorname{Tr}(e_P)\text{ mod }[\CH,\CH].
\end{equation}

\begin{example}For every subset $K\subsetneq S^a$, let $\CH_K\subset \CH$ be the (finite) parahoric Hecke algebra generated by $T_{s_i}$, $i\in K.$ If $\tau$ is a simple finite dimensional $\C H_K$-module, form the "compactly induced" module $P(K,\tau)=\CH\otimes_{\CH_K} \tau.$
When the parameters of $\CH$ are such that $\CH_K$ is a semisimple algebra, these modules are all finitely generated projective.
\end{example}

\begin{lemma}
The image of the rank map $\Rk$ lies in $\bar\CH^\rig.$
\end{lemma}

\begin{proof}
Let $P\in K_0(\CH)$ be a finitely generated projective module. By Theorem \ref{rigid-cocenter}, for $J\subset\Pi$, $(\sigma,V_\sigma)\in R(\CH_J^\sem)$ and $t\in \Hom(X\cap \bQ R/X\cap\bQ J,\bC^*)$, we need to show that $\tr(\Rk(P),i_J(\sigma_t))$ is independent of $t$, where $\sigma_t=\sigma\circ\chi_t$ are acting on the same vector space $V_\sigma$. By the adjunction property, $\tr(\Rk(P),i_J(\sigma\circ\chi_t))=\tr(\bar r_J(\Rk(P)),\sigma\circ\chi_t)$. Let $r_J: K_0(\CH)\to K_0(\CH_J)$ be the restriction functor, then $\bar r_J(\Rk(P))=\Rk(r_J(P)).$ Since $r_J(P)$ is a finitely generated projective $\CH_J$-module, let $e^J_P\in M_{n\times n}(\CH_J)$ be a corresponding idempotent. Thus, we have arrived at $\tr(\operatorname{Tr}(e^J_P),\sigma\circ\chi_t).$ 
This is equivalent with computing the matrix trace of the family of idempotents $\sigma_t(e^J_P)\in \End_\bC[V_\sigma]$. Now $\{\sigma_t(e^J_P)\}$ is a continuous family of idempotents, and by the rigidity of the trace of an idempotent, $\operatorname{Tr}(\sigma_t(e^J_P))$ is independent of $t$.
\end{proof}

\begin{remark}

One may regard this result as the affine Hecke algebra analogue of the "Selberg principle" for reductive $p$-adic groups of \cite{BB}, see \cite[Theorem 1.6]{Da}.
\end{remark}




\subsection{} We need a decomposition of $\bar\CH^\rig_\delta$ dual to the one for $\bar R^\delta(\CH)_\rig$ from Proposition \ref{p:dim-rigid} and Corollary \ref{c:rigid-section}.

\begin{proposition}\label{p:rigid-cocenter}
Suppose the root datum for $\CH$ is semisimple and that (PDT) holds. Then
$$\bar\CH^\rig_\delta=\bigoplus_{J\in \C I^\delta_\spadesuit} \bar i_J(\bar A_J({(\bar \CH_J^\sem)}^\rig)_{N_J}).$$
\end{proposition}

\begin{proof} The proof follows the lines of the proof of Proposition \ref{p:dim-rigid}. 

By Theorem \ref{rigid-cocenter}, $\bar\CH^{\EL}_\d\subseteq \bar\CH^{\rig}_\d$. By Proposition \ref{p:A}, $\bar A(\bar\CH^{\rig}_\d)=\bar \CH^{\EL}_\d$. Hence $\CH^{\EL}_\d=\bar A(\CH^{\EL}_\d)\subseteq \bar A(\CH^\rig_\d)\subseteq\bar A(\CH_\d)=\CH^{\EL}_\d.$ Thus $$\bar A(\bar\CH^\rig_\delta)=\bar \CH^{\EL}_\delta\subset \bar\CH^\rig_\delta.$$ 

Set $\bar p_0=\bar A \mid_{\bar\CH^\rig_\delta}$. Since $\bar A^2=a\bar A,$ $a\neq 0,$ by Proposition \ref{p:A-bar}, we have $\bar\CH^\rig_\delta=\bar A(\bar\CH^\rig_\delta)\oplus \ker\bar p_0.$

Since $\bar A=a+\sum_{J=\delta(J)\subsetneq \Pi} c_J' \bar \Delta_J$, $\ker \bar p_0\subset \sum_{J\in \CI^\delta_\spadesuit, J\neq\Pi} \bar i_J(\bar\CH_{\delta,J}).$ Recall that $\bar\CH^\rig_\delta$ is spanned by $T_\CO$, with $J_\CO=\Pi$. From Theorem \ref{BLbasis}, we see then that $$\ker\bar p_0\subseteq \sum_{J\in \CI^\delta_\spadesuit, J\neq \Pi} \bar i_J(\bar\CH^{\sem,\rig}_{\delta,J})_{N_J}.$$ Consider the projection
$$\bar p_1: \ker \bar p_0\to \bigoplus_{J\in \CI^\delta_\spadesuit, |J|=|\Pi|-1}\bar i_J(\bar A_J(\bar\CH^{\sem,\rig}_{\delta,J})_{N_J}).$$
By Theorem \ref{BLbasis}, the map is well-defined, i.e.,  the range of $\bar p_1$ is indeed a direct sum (rather than a sum). Moreover, $\bar p_1$ is surjective because the range of $\bar p_1$ is in $\bar \CH^\rig_\delta$, and it is orthogonal to $\bar\CH^{\EL}_\delta.$ Write $\ker \bar p_0=\ker\bar p_1\oplus\im \bar p_1$ and continue as in the proof of Proposition \ref{p:dim-rigid}.
\end{proof}

\

Now we state our main theorem. 

\begin{theorem}\label{t:rigid-pairing}
Suppose the root datum for $\CH$ is semisimple, the (PDT) holds\footnote{In section \ref{s:8}, the Density Theorem \ref{t:density} is proven, based on an inductive argument, and therefore, this assumption can be removed then.}, and the parameter $\bbq$ is admissible. Then 
\begin{enumerate}
\item $\tr:\bar\CH^\rig_\delta\times \bar R(\CH)_\rig^\delta\to\bC$ is a perfect pairing.
\item $\tr:\bar\CH^{\EL}_\delta\times \bar R^\delta_0(\CH)\to\bC$ is a perfect pairing.
\end{enumerate}
\end{theorem}

\begin{proof} 
We show that the pairing $\tr$ is block upper triangular with respect to the decomposition of $\bar\CH^\rig_\delta$ from Proposition \ref{p:rigid-cocenter} and the decomposition of $\bar R^\delta(\CH)_\rig$ from Corollary \ref{c:rigid-section}. For every $J\in\C I^\delta_\spadesuit$, denote 
$$\bar\CH^\rig_\delta[J] =\bar i_J(\bar A_J({(\bar \CH_J^\sem)}^\rig)_{N_J})\text{ and }\bar R^\delta(\CH)_\rig[J]=i_J(A_J(R(\CH_J^\sem))_\bC^{N_J}).$$
Let $J,K\in \CI^\delta_\spadesuit$ be given. Then, for $h\in (\bar \CH_K^{\sem})^\rig$ and $\pi\in R(\CH_J^\sem)$,
\begin{align*}
\tr(&\bar i_K(\bar A_K(h)),i_J(A_J(\pi)))=\tr(h, A_K\circ (r_K\circ i_J)(A_J(\pi))), &\text{by Lemma \ref{l:adjointtrace}},\\
&=\sum_{w\in ^KW^J(\delta)}\tr(h, A_K\circ(i_{K_w}^K\circ w\circ r_{J_w}^J)(A_J(\pi))), &\text{by Lemma \ref{l:ind-restr-2}}.
\end{align*}
In particular, if $J\not\supseteq K$ (as elements of $\CI^\delta_\spadesuit$), then all $K_w$ are proper subsets of $K$, and therefore $A_K$ kills the induced modules. This means that $\tr:\bar\CH^\rig_\delta[K]\times \bar R(\CH)_\rig^\delta[J]\to\bC$ is identically zero if $J\not\supseteq K.$ 

On the other hand, if $J=K$, the only nonzero contribution to the trace comes from $w\in~ ^KW^K(\delta)$ in the right hand side, and we obtain
\begin{equation*}
\tr(\bar i_K(\bar A_K(h)),i_K(A_K(\pi)))=|N_K|\tr(\bar A_K(h), A_K(\pi)).
\end{equation*}
This is because $A_K(\pi)$ is chosen to be $N_K$-invariant.

The rank of this pairing equals the rank of $\tr_K^\sem:(\bar \CH_{K,\delta}^{\sem,0})_{N_K}\times \bar R^\delta_0(\CH^\sem_K)^{N_K}\to\bC$, which, by Proposition \ref{p:finite-elliptic}, equals $\dim\bar R^\delta_0(\CH^\sem_K)^{N_K}$.

Therefore,
\begin{equation}
\begin{aligned}
\text{rank}\tr(\bar\CH^\rig_\delta\times \bar R(\CH)_\rig^\delta)&=\sum_{J\in\CI^\delta_\spadesuit}\text{rank}\tr(\bar\CH^\rig_\delta[J]\times \bar R(\CH)_\rig^\delta[J])\\
&=\sum_{J\in\CI^\delta_\spadesuit}\dim\bar R^\delta_0(\CH^\sem_J)^{N_J}\\
&=\dim \bar R(\CH)^\delta_\rig.\\
\end{aligned}
\end{equation}
This implies that $\dim\bar\CH^\rig_\delta\ge \dim \bar R(\CH)^\delta_\rig$. On the other hand, from Theorem \ref{rigid-cocenter}, we know that $\dim \bar\CH^\rig_\delta \le |cl(\wti W,\delta)_0|$. When $\bbq$ is admissible, $\dim\bar R(\CH)^\delta_\rig=|cl(\wti W,\delta)_0|$, and therefore $\dim\bar\CH^\rig_\delta=\dim \bar R(\CH)^\delta_\rig=|cl(\wti W,\delta)_0|$, in this case. This proves claim (1).

In light of the inequality in Proposition \ref{p:finite-elliptic} (applied to each $J$), (1) implies that the equality of dimensions must hold for each $J$, i.e., $\dim\bar\CH^\rig_\delta[J]=\dim\bar R^\delta(\CH)_\rig[J]$. In particular, when $J=\Pi$, this is claim (2).
\end{proof}

\subsection{Examples} In this subsection, we illustrate the pairing between the rigid quotient and the rigid cocenter in some concrete examples.

\begin{example}[$SL(2)$] Let $\CH$ be the affine Hecke algebra for the root datum of $SL(2).$ This is generated by $T_0$ and $T_1$ subject to $$T_i^2=(\bbq-1)T_i+\bbq,\quad i=0,1.$$
There are three conjugacy classes of finite order in $\wti W$: $s_0$ (the affine reflection), $s_1$ (the finite reflection), and $1$. Accordingly, we have the three basis elements of $\bar\CH^\rig$: $T_0$, $T_1$, and $1$.

There are four one-dimensional modules corresponding to $T_i\in\{-1,\bbq\}.$ A basis of the elliptic space $\bar R_0(\CH)$ in the admissible case is given by the classes of $\St=(T_0=-1,T_1=-1)$ and any one of the two modules $\pi^+=(T_0=-1, T_1=q)$ or $\pi^-=(T_0=q, T_1=-1).$ Choose the class of $\pi^+$.

The third basis element of $\bar R(\CH)_\rig$ is the (tempered) principal series $i_\emptyset(1)$. The resulting $3\times 3$ table is in Table \ref{ta:sl2}.

\begin{table}
\caption{The rigid cocenter/quotient table for $SL(2)$\label{ta:sl2}}
\begin{tabular}{|c|c|c|c|}
\hline
$SL(2)$&$\St$&$\pi^+$&$i_\emptyset(1)$\\
\hline
$T_0$ &$-1$ &$-1$ &$\bbq-1$\\
\hline
$T_1$ &$-1$ &$\bbq$ &$\bbq-1$\\
\hline
$1$ &$1$ &$1$ &$2$\\
\hline
\end{tabular}
\end{table}

The determinant of this matrix is $-(\bbq+1)^2$, and therefore the matrix is invertible for all $\bbq\neq -1.$

To get a block upper-triangular matrix, one needs to replace $T_i$  by $\bar A(T_i)$, $i=0,1$, as in Theorem \ref{t:rigid-pairing}.
\end{example}

\begin{example}[$PGL(2)$] Let $\CH$ be the affine Hecke algebra for the root datum of $PGL(2).$ This is generated by $T_0,T_1,\tau$ subject to $$T_i^2=(\bbq-1)T_i+\bbq,\quad i=0,1,\quad \tau^2=1,\quad \tau\cdot T_0=T_1\cdot\tau.$$
There are three conjugacy classes of finite order in $\wti W$:  $s_1$ (the finite reflection), $\tau$, and $1$. Accordingly, we have the three basis elements of $\bar\CH^\rig$: $T_1$, $\tau$, and $1$.

There are four one-dimensional modules corresponding to $\tau\in\{1,-1\}$ and $T_1\in\{-1,\bbq\}.$ A basis of the elliptic space $\bar R_0(\CH)$ is given by the classes of $\St^\pm=(T_1=-1,\tau=\pm 1)$. 

The third basis element of $\bar R(\CH)_\rig$ is the (tempered) principal series $i_\emptyset(1)$. The resulting $3\times 3$ table is in Table \ref{ta:pgl2}.

\begin{table}
\caption{The rigid cocenter/quotient table for $PGL(2)$\label{ta:pgl2}}
\begin{tabular}{|c|c|c|c|}
\hline
$PGL(2)$&$\St^-$&$\St^+$&$i_\emptyset(1)$\\
\hline
$T_1$ &$-1$ &$-1$ &$\bbq-1$\\
\hline
$\tau$ &$1$ &$-1$ &$0$\\
\hline
$1$ &$1$ &$1$ &$2$\\
\hline
\end{tabular}
\end{table}

The determinant of this matrix is $2 (\bbq+1)$, and therefore the matrix is invertible for all $\bbq \neq -1$.
\end{example}

\begin{example}[Affine $C_2$] Let $\CH$ be the affine Hecke algebra attached to the affine diagram of type $C_2$ with three parameters
\begin{equation}\label{c2-aff}
\xymatrix{\bbq_0&\bbq_1\ar@{<=}[l]&\bbq_2\ar@{=>}[l]},
\end{equation}

It is generated by $T_i$, $i=0,1,2$, subject to $T_i^2=(\bbq_i-1)T_i+\bbq_i$ and the braid relations.

There are nine conjugacy classes of finite order in $\wti W$, including five elliptic classes. Representatives for the five elliptic classes are: $s_1s_2$, $(s_1s_2)^2$, $s_0s_2$, $s_0s_1$, and $(s_0s_1)^2$. The other four classes correspond to: $s_0$, $s_1$, $s_2$, and $1$.

A basis of the elliptic space $\bar R_0(\CH)$ in the admissible case can be constructed by lifting the five simple modules for the finite Hecke algebra $\CH_f(C_2,\bbq_1,\bbq_2)$ under the algebra homomorphism $T_0\mapsto -1$, $T_1\mapsto T_1$, $T_2\mapsto T_2.$ We label these five modules by the bipartitions which parameterized the corresponding representations of the finite Weyl group: $2\times 0$, $11\times 0$, $0\times 2$, $0\times 11$, and $1\times 1.$ For the character table of the finite Hecke algebra of type $C_2$, see \cite{GP}.

The remaining four modules needed for a basis of the rigid quotient $\bar R(\CH)_\rig$ can be chosen as the induced tempered modules: $i_{\{1\}}(\St)$, $i_{\{2\}}(\St)$, $i_{\{2\}}(\pi^+)$, and $i_\emptyset(1).$ Here $\pi^+$ is the one dimensional elliptic module for the Hecke algebra of $SL(2)=Sp(2)$ as in Table \ref{ta:sl2}. 

For this calculation, we computed the restrictions of the $9$ representations to the maximal finite Hecke subalgebras, using the Mackey formula, and then used the character table for the finite Hecke algebras of type $C_2$ and $A_1$. 
The resulting rigid character table is given, because of page limitations, in (\ref{C2}) and (\ref{C2-1}).

\begin{equation}\label{C2}
\begin{tabular}{|c|c|c|c|c|c|}
\hline
$C_2^{\mathsf{aff}}$&$2\times 0$ &$11\times 0$ &$0\times 2$ &$0\times 11$ &$1\times 1$ \\
\hline
$T_1T_2$ &$\bbq_1\bbq_2$ &$-\bbq_2$ &$-\bbq_1$ &$1$ &$0$  \\
\hline
$(T_1T_2)^2$ &$(\bbq_1\bbq_2)^2$ &$\bbq_2^2$ &$\bbq_1^2$ &$1$ &$-2\bbq_1\bbq_2$ \\
\hline
$T_0T_2$ &$-\bbq_2$ &$-\bbq_2$ &$1$ &$1$ &$1-\bbq_2$  \\
\hline
$T_0T_1$ &$-\bbq_1$ &$1$ &$-\bbq_1$ &$1$ &$1-\bbq_1$  \\
\hline
$(T_0T_1)^2$ &$\bbq_1^2$ &$1$ &$\bbq_1^2$ &$1$ &$\bbq_1^2+1$   \\
\hline
$T_0$ &$-1$ &$-1$ &$-1$ &$-1$ &$-2$ \\
\hline
$T_1$ &$\bbq_1$ &$-1$ &$\bbq_1$ &$-1$ &$\bbq_1-1$ \\
\hline
$T_2$ &$\bbq_2$ &$\bbq_2$ &$-1$ &$-1$ &$\bbq_2-1$ \\
\hline
$1$ &$1$ &$1$ &$1$ &$1$ &$2$ \\
\hline
\end{tabular}
\end{equation}

\begin{equation}\label{C2-1}
\begin{tabular}{|c|c|c|c|c|}
\hline
$C_2^{\mathsf{aff}}$ &$i_{\{1\}}(\St)$ &$i_{\{2\}}(\St)$ &$i_{\{2\}}(\pi^+)$ &$i_\emptyset(1)$ \\
\hline
$T_1T_2$  &$1-\bbq_2$ &$1-\bbq_1$ &$\bbq_2(\bbq_1-1)$ &$(\bbq_1-1)(\bbq_2-1)$ \\
\hline
$(T_1T_2)^2$ &$\bbq_2^2-2\bbq_1\bbq_2+1$ &$\bbq_1^2-2\bbq_1\bbq_2+1$ &$\bbq_1^2\bbq_2^2+\bbq_2^2-2\bbq_1\bbq_2$ &$\bbq_1^2\bbq_2^2+\bbq_2^2+\bbq_1^2+1-4\bbq_1\bbq_2$\\
\hline
$T_0T_2$ &$(\bbq_0-1)(\bbq_2-1)$ &$2-\bbq_0-\bbq_2$ &$\bbq_0\bbq_2+1-2\bbq_2$ &$2(\bbq_0-1)(\bbq_2-1)$\\
\hline
$T_0T_1$  &$1-\bbq_0$ &$1-\bbq_1$ &$1-\bbq_1$ &$(\bbq_0-1)(\bbq_1-1)$\\
\hline
$(T_0T_1)^2$  &$\bbq_0^2-2\bbq_0\bbq_1+1$ &$\bbq_1^2-2\bbq_0\bbq_1+1$ &$\bbq_1^2-2\bbq_0\bbq_1+1$ &$\bbq_0^2\bbq_1^2+\bbq_1^2+\bbq_0^2+1-4\bbq_0\bbq_1$\\
\hline
$T_0$ &$2\bbq_0-2$ &$\bbq_0-3$ &$\bbq_0-3$ &$4\bbq_0-4$\\
\hline
$T_1$  &$\bbq_1-3$ &$2\bbq_1-2$ &$2\bbq_1-2$ &$4\bbq_1-4$\\
\hline
$T_2$  &$2\bbq_2-2$ &$\bbq_2-3$ &$3\bbq_2-1$ &$4\bbq_2-4$\\
\hline
$1$  &$4$ &$4$ &$4$ &$8$\\
\hline
\end{tabular}
\end{equation}

The determinant of the $9\times 9$ rigid table equals
$$-(1+\bbq_0)^3(1+\bbq_1)^3(1+\bbq_2)^3(\bbq_0+\bbq_1)(\bbq_1+\bbq_2)(1+\bbq_0\bbq_1)(1+\bbq_1\bbq_2).$$
This implies that the determinant of the rigid table above is nonzero if and only if all the finite Hecke algebras are semisimple\footnote{For the semisimplity criterion for the finite Hecke algebra of type $B_n$, see \cite[Theorem 5.5]{DP}}.

\smallskip

\end{example}

\begin{example}[Extended affine $C_2$] The Hecke algebra attached to the extended affine diagram of type $C_2$
\begin{equation}\label{h1}
\xymatrix{\bbq_2&\bbq_1\ar@{<=}[l]&\bbq_2\ar@{=>}[l]\ar@{<->}@/_1pc/[ll]}
\end{equation}
is isomorphic to the Hecke algebra attached to the affine diagram
\begin{equation}\label{h2}
\xymatrix{1&\bbq_2\ar@{<=}[l]&\bbq_1\ar@{=>}[l]},
\end{equation}
so this is a particular case of the previous example with the appropriate specialization of the parameters. More precisely, let $T_0,T_1,T_2,\tau$ be the generators of the Hecke algebra attached to the diagram (\ref{h1}) with relations: 
\begin{align*}
&T_0^2=(\bbq_2-1)T_0+\bbq_2, &&T_1^2=(\bbq_1-1)T_1+\bbq_1, &&T_2^2=(\bbq_2-1)T_2+\bbq_2,\\
&\tau^2=1, &&\tau T_0= T_2\tau,  &&\tau T_1= T_1\tau,\\
&T_0T_1T_0T_1=T_1T_0T_1T_0, &&T_1T_2T_1T_2=T_2T_1T_2T_1, &&T_0T_2=T_2T_0, 
\end{align*}
and let $T_0',T_1',T_2'$ be the generators of the Hecke algebra attached to the diagram (\ref{h2}) with relations:
\begin{align*}
&(T_0')^2=(\bbq_1-1)T_0'+\bbq_1, &&(T_1')^2=(\bbq_2-1)T_1'+\bbq_2, &&(T_2')^2=1,\\
&T_0'T_1'T_0'T_1'=T_1'T_0'T_1'T_0', && T_1'T_2'T_1'T_2'=T_2'T_1'T_2'T_1',&&T_0'T_2'=T_2'T_0'.
\end{align*}
The isomorphism is realized by 
\begin{equation}
\tau\longleftrightarrow T_0', \quad T_0\longleftrightarrow T_0'T_1'T_0',\quad T_1\longleftrightarrow T_2',\quad T_2\longleftrightarrow T_1'.
\end{equation}
The determinant of the rigid table equals (up to a scalar independent of $\bbq_1,\bbq_2$):
$$(1+\bbq_1)^3(1+\bbq_2)^5(\bbq_1+\bbq_2)(1+\bbq_1\bbq_2).$$
\end{example}

\section{Some consequences}\label{s:8}

In this section, we prove the basis theorem, density theorem and trace Paley-Wiener theorem. They will be proven inductively, using Theorem \ref{t:rigid-pairing}.

\begin{theorem}[Basis Theorem for arbitrary parameters]\label{t:basis} $\bar \CH'$ is a free $\Lambda$-module with basis $\{T_\co\}$, where $\co$ ranges over all the conjugacy classes of $\wti W'$.
\end{theorem}

\begin{theorem}[Density Theorem for admissible parameters]\label{t:density}
When the parameter $\bbq$ is admissible, the  trace map $\tr: \bar \CH' \to R^*(\CH')_\good$ is bijective.
\end{theorem}

The results in the previous subsections regarding $\bar R^\delta(\CH)_\rig$ and $\bar\CH^\rig_\delta$ were proved under the assumption that the root datum is semisimple and the parameters are admissible. In order to apply an inductive argument using parabolic subalgebras, we need to show that a proof in the semisimple case is sufficient to derive the general case. First we need some elements of Clifford theory.

\subsection{}\label{s:reduction}
Suppose that $A$ is an associative algebra and $\Gamma$ a group acting on $A$. Set $A'=A\rtimes \bC[\Gamma].$ The map $a\gamma\mapsto  \gamma$, $a\in A,\ \gamma\in \Gamma$ induces a surjective linear map
\begin{equation}
\tau: \bar A'\to \overline{\bC[\Gamma]}=\bigoplus_{[\gamma]\in cl(\Gamma)}\bC\gamma.
\end{equation}
The fiber $\tau^{-1}(\gamma)$ is by definition the image of $A\gamma$ in $\bar A'$. Similar to section \ref{s:d-comm}, we have:
\begin{lemma}\label{l:clifford}
The map $a\gamma\mapsto a$ induces an isomorphism $\tau^{-1}(\gamma)\cong \bar A^{[\gamma]}$, where $\bar A^{[\gamma]}=(A/[A,A]_\gamma)_{Z_\Gamma(\gamma)}$, the space of $Z_\Gamma(\gamma)$-coinvariants.
\end{lemma}

\begin{proposition}\label{reduction-ss}
Assume the root datum is semisimple and that $\bbq$ is admissible. If the statement of Theorem \ref{t:basis} holds when $\CH'$ is semisimple and $\bbq$ is admissible, then it holds in general.
\end{proposition}

\begin{proof}
Define the linear map $$\pi: \bigoplus_{\co\in cl(\wti W')}\Lambda\to \bar\CH',\ (a_\co)\mapsto \sum a_\co T_\co.$$
By Theorem \ref{IMbasis}, this map is surjective. We need to show also that $\ker\pi=0.$ For this, it is sufficient to show that $\ker\pi\otimes_\Lambda \bC_q=0$ for generic parameters $q.$ Fix now an admissible $q$ and specialize $\CH'_q=\CH'\otimes_\Lambda\bC_q.$ We prove that $\{T_\co: \co\in cl(\wti W')\}$ are $\bC$-linear independent in $\bar\CH'_q$.

Let $\pi^q:\bigoplus_{\co\in cl(\wti W')}\bC\to \bar\CH'_q$ denote the map induced by $\pi$ after specialization.

Rewrite $\wti W'=\wti W\rtimes \Gamma$ as $\wti W'=W^a\rtimes \wti \Gamma,$ where $\wti\Gamma=\Omega\rtimes\Gamma.$ For every $\co\in cl(\wti W')$, there exists a {unique} $[\gamma]\in cl(\wti\Gamma)$ such that $\co\cap W^a \gamma\neq \emptyset.$ By Lemma \ref{l:clifford}, it is sufficient to prove that for a fixed $[\gamma]\in cl(\wti\Gamma)$, the set  $$\{T_\co: \co\in cl(\wti W'), \co\cap W^a\gamma\neq\emptyset\}$$ is linearly independent. 

Denote $\underline{\wti\Gamma}$ the image of $Z_\Gamma(\gamma)$ in $\operatorname{Aut}(W^a,S^a)$, and by $\underline\gamma$ the image of $\gamma.$ Moreover, set $\underline{\wti W}'=W^a\rtimes \underline{\wti\Gamma}$ and $\underline{\wti\CH}'=\CH_a\rtimes \bC[\underline{\wti \Gamma}].$ Here $\CH_a$ is the affine Hecke algebra corresponding to $W^a$. Notice that $\underline{\wti\CH}'$ is an affine Hecke algebra attached to a semisimple root datum and extended by a group of automorphisms, therefore the hypothesis of the Proposition applies to it.

Let $$\underline\pi^q:\bigoplus_{\underline\CO\in cl(\underline{\wti W}')}\bC\to \overline{\underline{\wti\CH}}'_q,\quad (a_{\underline\CO})\mapsto \sum a_{\underline\CO}T_{\underline\CO},$$
be the cocenter map. By the assumption in the Proposition, $\underline\pi^q$ is bijective. Consider its restriction to the $\gamma$-fiber:
\begin{equation}\label{gamma-fiber}
\underline{\pi}^q_{[\underline\gamma]}:\bigoplus_{\underline\CO\in cl(\underline{\wti W}'): \underline\CO\cap W^a\underline\gamma\neq\emptyset}\bC\to \text{Im}(\CH_{a,q}\underline\gamma\to \overline{\underline{\wti\CH}}'_q )\cong \overline\CH_{a,q}^{[\underline\gamma]},
\end{equation}
where the last isomorphism is by Lemma \ref{l:clifford}. Since $\underline\pi^q$ is injective, so is $\underline{\pi}^q_{[\underline\gamma]}$.

There is a natural bijection:
\begin{equation}
\begin{aligned}
\kappa_\gamma:\{\co:\co\in cl(\wti W'), \co\cap W^a\gamma\neq\emptyset\}&\longrightarrow \{\underline\co: \underline\co\in cl(\underline{\wti W}'), \underline\co\cap W^a\underline\gamma\neq\emptyset\}\\
[w\gamma]&\longrightarrow   [w\underline\gamma], \quad w\in W^a.\\
\end{aligned}
\end{equation}

Similar to (\ref{gamma-fiber}), define
\begin{equation}
\pi^q_{[\gamma]}:\bigoplus_{\CO\in cl({\wti W}'):\CO\cap W^a\gamma\neq\emptyset}\bC\to \text{Im}(\CH_{a,q}\gamma\to \overline{{\CH}}'_q )\cong \overline\CH_{a,q}^{[\gamma]}.
\end{equation}
We wish to show that $\pi^q_{[\gamma]}$ is also injective, and this would complete the proof. We have the commutative diagram:
\begin{equation}\displaystyle{
\xymatrix{\bigoplus_{\CO\in cl({\wti W}'):\CO\cap W^a\gamma\neq\emptyset}\bC \ar[d]_{\pi^q_{[\gamma]}} \ar[r]^{\kappa_\gamma} &  \bigoplus_{\underline\CO\in cl(\underline{\wti W}'): \underline\CO\cap W^a\underline\gamma\neq\emptyset}\bC \ar[d]^{\underline{\pi}^q_{[\underline\gamma]}}\\
\overline\CH_{a,q}^{[\gamma]} \ar[r]^{\cong} &\overline\CH_{a,q}^{[\underline\gamma]}}
}
\end{equation}
Since $\underline{\pi}^q_{[\underline\gamma]}$ is injective and $\kappa_\gamma$ is bijective, it follows that $\pi^q_{[\gamma]}$ is injective.

\end{proof}

\subsection{Proofs of Theorems \ref{t:basis} and \ref{t:density}}\label{s:proofs} In light of Proposition \ref{reduction-ss}, it is sufficient to prove the theorems under the assumption that the root datum is semisimple. Assume that the indeterminate $\bbq$ is specialized to an admissible parameter $q$. By induction, we may assume that both the basis and density theorems hold for all proper parabolic subalgebras.  

Suppose $\sum_\co a_\co T_\co \in \ker \tr$. We claim that by induction, $\sum_\co a_\co T_\co  \in \bar\CH^{\rig}_\d$. Indeed, by the proof of Theorem \ref{rigid-cocenter} and the (PDT) assumption, $\sum_{J_\co=J} a_\co T_\co=0$ for all $J \subsetneqq \Pi$. By the inductive assumption on proper parabolic subalgebras, $a_\co=0$ for all $\co$ with $J_\co \neq J$. 

Thus $\sum_{\co \in cl(\wti W, \d)_0} a_\co T_\co \in \ker \tr$. Note that $\bar\CH^\rig_\d$ is spanned by $T_\co$ for $\co \in cl(\wti W, \d)_0$ and $\dim \bar R^\d_0(\CH)_\bC=|cl(\wti W, \d)_0|$. By Theorem \ref{t:rigid-pairing}, $T_\co$ for $\co \in cl(\wti W, \d)_0$ forms a basis of $\bar\CH^\rig_\d$. Hence $a_\co=0$ for $\co \in cl(\wti W, \d)$. This concludes the proof when $\bbq$ is admissible.

\smallskip

Finally, since Theorem \ref{t:basis} holds for generic parameters, it holds for the indeterminate $\bbq$ as well.
\qed

\subsection{}\label{8.3}
Since $\bar \CH_\d$ is a $\Lambda$-module with basis $T_\co$, it is easy to see that the induction and restriction functors $\bar i_J$ and $\bar r_J$ depend algebraically on the parameters, therefore Lemma \ref{l:ind-restr-bar} and hence Proposition \ref{p:A-bar} and Proposition \ref{p:finite-elliptic} hold unconditionally. Moreover, $\bar \CH^{\EL}_\d \subset \bar \CH^\rig_\d$ and Proposition \ref{p:rigid-cocenter} holds without the (PDT) assumption. 

We can now prove the trace Paley-Wiener theorem for arbitrary parameters.

\begin{theorem}\label{P-W1}
For arbitrary parameters, the image of the map $\tr: \bar \CH' \to R(\CH')$ is  $R^*(\CH')_\good.$ In other words, $R^*_\delta(\CH)_{\tr}=R^*_\delta(\CH)_\good.$
\end{theorem}

\begin{proof}
We first show that \[\tag{a} \tr:\bar\CH^{\EL}_\d\to \bar R_{0}^\d(\CH)^* \text{ is surjective.} \]

By Proposition \ref{p:A}, we may regard $\bar R_0^\d(\CH)$ as $A(R^\d(\CH))\subset R^\d(\CH).$ Under this identification, $\bar R_0^\d(\CH)=\{\s \circ \chi_t; \s \in \bar R_0^\d(\CH^{\sem}), t \in T_\Pi^\d\}$. On the other hand, the natural projection map $\bar \CH_\d \to \bar (\CH^{\sem})_\d$ is surjective and each fiber is isomorphic to $\oplus_{\lambda \in (\Pi^\vee)^\perp/\<\d\>} \bC \theta_\lambda h$ for some $h \in \bar \CH_\d$. Note that $\oplus_{\lambda \in (\Pi^\vee)^\perp/\<\d\>} \bC \theta_\lambda$ is naturally isomoprhic to the set of regular functions on $T_\Pi^\d$. Thus (a) follows from Proposition \ref{p:finite-elliptic} and the remark above for $\CH^\sem$. 


The rest of the argument is as in \cite[\S4]{BDK}. Let $f\in R^*_\d(\CH)_\good$ be given. By (a), there exists $h\in \bar\CH^{\EL}_\d$ such that $f'_h=f$ on $\bar R_0^\d(\CH),$ where $f'_h:=\tr(h)\in R^*_\d(\CH)_{\tr}.$ Modifying $f$ by $f-f'_h,$ we may assume that $f(\bar R_0^\d(\CH))=0.$

Consider the adjoint operator $$A^*=a+\sum_{J=\delta(J)\subsetneq \Pi} c_J' i_J^* \circ r_J^*: R^*_\delta(\CH)\to R^*_\delta(\CH).$$ Then $A^*(f)(R^\d(\CH))=f(\bar R_0^\d(\CH))=0$. So
$f=-\sum_{J=\delta(J)\subsetneq \Pi}c'_J/a ~ i_J^* \circ r_J^*(f).$
It is immediate that $i_J^*(R^*_\delta(\CH)_\good)\subset R^*_\delta(\CH_J)_\good$ and $r_J^*(R^*_\delta(\CH_J)_{\tr})\subset R^*_\delta(\CH)_{\tr}$, thus the claim follows by induction on $J$.
\end{proof}

\subsection{} We record one more finiteness result for arbitrary parameters.

\begin{proposition}\label{p:rank-H0}
Suppose that $\Phi$ is a semisimple based root datum. Then $$\dim \bar R^\delta_0(\CH) \le \rank \bar \CH^{\EL}_\d=|cl(\wti W, \d)_{\EL}|.$$ 
\end{proposition}

\begin{proof}
By Proposition \ref{p:finite-elliptic} and \S \ref{8.3}, $\dim \bar R^\d_0(\CH)\le \dim \bar\CH^{\EL}_\d$ for arbitrary parameters. Since (PDT) holds for admissible parameters,  $\rank \bar \CH^{\EL}_\d=\rank \bar \CH_\d/\ker \bar A$ is finite for admissible parameters. Since $\bar \CH_\d$ is a $\Lambda$-module with basis $T_\co$ and the map $\bar A$ depends algebraically on the parameters, $\rank \bar \CH_\d/\ker \bar A$ is semi-continuous, i.e., the rank at any parameter is less than or equal to the rank at generic parameters. By Theorem \ref{t:rigid-pairing} (2) and Proposition \ref{elliptic} (2), $\rank \bar \CH^{\EL}_\d \le |cl(\wti W, \d)_{\EL}|$. 

On the other hand, by Proposition \ref{p:rigid-cocenter}, $$\rank \bar\CH^\rig_\delta=\sum_{J\in \C I^\delta_\spadesuit} \rank \overline{(\CH_J^\sem)}^0_\d/N_J.$$ For each $J$, the $N_J$ action depends algebraically on the parameters, hence $$\rank \overline{(\CH_J^\sem)}^0_\d/N_J \le |cl(\wti W_J, \d)_{\EL}/N_J|.$$ By \S\ref{e:HN2-dims}, $$|cl(\wti W, \d)_0|=\rank \bar\CH^\rig_\delta=\sum_{J\in \C I^\delta_\spadesuit} \rank \overline{(\CH_J^\sem)}^0_\d/N_J \le \sum_{J\in \C I^\delta_\spadesuit} |cl(\wti W_J, \d)_{\EL}/N_J|=|cl(\wti W, \d)_0|.$$ Thus $\rank \overline{(\CH_J^\sem)}^0_\d/N_J=|cl(\wti W_J, \d)_{\EL}/N_J|$ for all $J$. In particular, $\rank \bar \CH^{\EL}_\d=|cl(\wti W, \d)_{\EL}|$. 
\end{proof}

\begin{remark}
If the root datum $\Phi$ is not semisimple, apply the constructions from section \ref{s:induction} with $J=\Pi.$ Set $Z(\Phi)=X/X\cap\bQ\Pi$ and $\widehat {Z(\Phi)}=T^\Pi=\Hom_\bZ(Z(\Phi),\bC^\times).$ For every $t\in \widehat{Z(\Phi)}$, let $\chi_t: \CH\to\CH^\sem$ be the algebra homomorphism (\ref{chi-t}). These homomorphisms define a right $\widehat{Z(\Phi)}$-action on $\Irr^\delta\CH$ and on $\Theta(\CH)$, which preserves the elliptic parts $\Irr^\delta_0\CH$ and $\Theta(\CH)_0.$ Then Proposition \ref{p:rank-H0} implies that the sets
$\Theta^\delta(\CH)_0$ and $\Irr^\delta_0(\CH)$ are finite unions of $\widehat{Z(\Phi)}$-orbits. \end{remark}

\section{Applications to smooth representations of reductive $p$-adic groups}\label{s:unitary}

\subsection{} Let $k$ be a $p$-adic field of characteristic $0$ with ring of integers $\fk o$ and residue field of cardinality $q$. Let $\C G$ be a connected semisimple group over $k$, and set $G=\C G(k).$ Let $G^0$ be the subgroup of $G$ generated by all of the compact open subgroups. A (one-dimensional) smooth character of $G$ is called unramified if $\chi|_{G^0}=1.$ 
Let $X_u(G)$ denote the group of unramified characters of $G$. 
Let $R(G)$ denote the Grothendieck group of smooth $\bC$-representations of $G$. Let $\C B(G)$ be the (modified) Bernstein center of $G$, as in \cite{BK}. Then we have a decomposition $R(G)=\prod_{\fk s\in \C B(G)} R^{\fk s}(G).$ 

If $P$ is a ($k$-rational) parabolic subgroup of $G$, with Levi decomposition $P=MN$, let $i_M:R(M)\to R(G)$ and $r_M:R(G)\to R(M)$ denote the functor of normalized parabolic induction and the normalized Jacquet functor, respectively, see for example \cite{BDK}, \cite{Be}.

Analogous with the Hecke algebras definition \ref{d:Hecke-rigid}, we define the rigid quotient of $R(G).$

\begin{definition}\label{d:group-rigid}
Set
\begin{equation}
\begin{aligned}
R(G)_{\twist}=\text{span}\{i_M(\sigma)-i_M(\sigma\circ\chi):&~ P=MN\text{ parabolic subgroup,}\\
&~\sigma\in
R(M),~ \chi\in X_u(M) \}
\end{aligned}
\end{equation}
and define the rigid quotient of $R(G)$ to be
\begin{equation}
\bar R(G)_{\rig}:=R(G)/R(G)_{\twist}.
\end{equation}
If $\fk s\in \C B(G)$, let $\bar R^{\fk s}(G)_\rig$ be the image of the Bernstein component $R^{\fk s}(G)$ in $\bar R(G)_\rig.$
\end{definition}

\subsection{} We explain the role of the rigid cocenter in the correspondence of unitarizable representations from the $p$-adic group to Iwahori-Hecke algebras. Let $G$ satisfy the same assumptions as in \cite{Bo}.

Let $I$ be an Iwahori subgroup of $G$. Let $\C C_I(G)$ be the category of smooth $G$-representations generated by their $I$-fixed vectors. By a classical result of Casselman, 
$$\C C_I(G)=R^0(G),$$
where $R^0(G)$ is the Bernstein component where the simple objects are sub quotients of unramified minimal principal series.

Let $\CH=C^\infty_c(I\backslash G/I)$ be the Iwahori-Hecke algebra, i.e., the complex associative unital algebra of compactly supported, smooth, $I$-biinvariant complex functions, under the convolution with respect to a fixed Haar measure $\mu$ on $G$. Normalize the Haar measure so that $\mu(I)=1.$ Given a representation $(\pi,V)\in \C C_I(G),$ the algebra $\CH$ acts on $V^I$ via
$$\pi(f)v=\int_G f(x)\pi(x)v~d\mu(x),\quad f\in\CH,\ v\in V^I.$$
Borel \cite{Bo} showed that the functor 
\begin{equation}\label{borel}
F_I:~\C C_I(G)\to \CH\text{-mod},\quad V\mapsto V^I
\end{equation}
is an equivalence of categories.

A smooth admissible $G$-representation $(\pi,V)$ is called hermitian if $V$ has a non degenerate hermitian form $\langle~,~\rangle_G$ which is $G$-invariant, i.e., $\langle\pi(g)v,w\rangle_G=\langle v,\pi(g^{-1})w\rangle_G,$ for all $g\in G$, $v,w\in V$. A hermitian representation is called (pre)unitary if the hermitian form is positive definite.

In parallel, we have the similar definitions for $\C H$-modules. Let $*:\C H\to \C H$ denote the conjugate-linear anti-involution given by $f^*(g)=\overline{f(g^{-1})},$ $f\in \C H,$ $g\in G.$ Then an $\CH$-invariant hermitian form on a $\CH$-module $(\pi, V^I)$ has the defining property $\langle \pi(f) v, w\rangle_{\C H}=\langle v, \pi(f^*)w\rangle_{\C H}.$

By restriction to $I$-fixed vectors, it is clear that if $(\pi,V)\in \C C_I(G)$ is hermitian (resp., unitary), then $(\pi,V^I)\in \CH\text{-mod}$ is hermitian (reps., unitary). It is also easy to see that if $(\pi,V^I)$ is hermitian, then $(\pi,V)$ is hermitian, see \cite[(2.10)]{BM1}. The difficult implication is the following:

\begin{theorem} Let $(\pi,V)\in \C C_I(G)$ be an irreducible hermitian $G$-representation. If $V^I$ is a unitary $\CH$-module, then $V$ is a unitary $G$-representation.
\end{theorem}

The above theorem was proven in \cite{BM1} under the assumption that $G$  is split of adjoint type and that $V$ has "real infinitesimal character". The latter assumption was removed in \cite{BM2} using a reduction to graded Hecke algebras and endoscopic groups. The assumption that $G$ be split adjoint was removed in \cite{BC} (where the theorem was generalized to other Bernstein components as well), but the argument still involved the reduction to real infinitesimal character. In the next subsection, we explain that the present result on the rigid cocenter, Theorem \ref{t:rigid-pairing}(1), allows for a direct extension of the \cite{BM1} argument to arbitrary representations $V$ with Iwahori fixed vectors for any semisimple connected $k$-group $G$ in the sense of \cite{Bo}. This allows one to bypass the reduction to real infinitesimal character from \cite{BM2,BC}.

\subsection{} 
Fix a basis $\{\bar V_i\}_{1\le i \le n}$ of $\bar R(\C H)_\rig$ consisting of genuine finite-dimensional $\C H$-modules. Invoking Langlands classification or Corollary \ref{c:rigid-section}, it is clear that we may choose $\bar V_i$ so that they are all tempered $\CH$-modules. Under the functor (\ref{borel}), there exist smooth admissible $G$-representation $V_i$ such that $V_i^I=\bar V_i.$ It is well-known that $F_I$ induces a bijection between tempered representations in the two categories, thus $V_i$ are tempered $G$-representations. In fact, $F_I$ induces a homeomorphism between the supports of the Plancherel measures in the two categories, see \cite{BHK}, but we will not need this more precise statement.

Since the functor (\ref{borel}) commutes with parabolic induction, see the discussion around \cite[Theorem 6.1]{BM1}, we immediately have that $\{V_i\}$ are a basis of $\bar R^0(G)_\rig$. 

Let $(\pi,V)\in \C C_I(G)$ be an irreducible, hermitian representation with $G$-invariant hermitian form $\langle~,~\rangle_G$. Let $K\supset I$ be a maximal parahoric subgroup. Following \cite{Vo}, one defines the $K$-signature of $V$ as follows. Let $(\sigma,U_\sigma)$ be a smooth irreducible $K$-representation with an implicitly fixed positive definite $K$-invariant form. The $K$-character of $V$ is the formal sum
\begin{equation}
\theta_K(V)=\sum_{\sigma\in \widehat K} m_\sigma\sigma,\quad \text{where }m_\sigma=\dim\Hom_K[U_\sigma,V]<\infty.
\end{equation}
Then the form $\langle~,~\rangle_G$ induces a form on the isotopic component $\Hom_K[U_\sigma,V]$ of $\sigma$ in $V$ (if this is nonzero), and this form has signature $(p_\sigma,q_\sigma).$ Of course, $p_\sigma+q_\sigma=m_\sigma.$ The $K$-signature character of $V$ is the formal combination
\begin{equation}
\Sigma_K(V)=(\sum_{\sigma\in\widehat K} p_\sigma \sigma, \sum_{\sigma\in\widehat K} q_\sigma \sigma).
\end{equation}

Since $\{V_i\}$ is a basis of $\bar R^0(G)_\rig$, in particular, we have the following sharpening of \cite[Theorem 5.3]{BM1}.

\begin{lemma}
For every $K\supset I,$ the $K$-character of any irreducible representation in $\C C_I(G)$ is a linear combination of $\theta_K(V_i)$, $i=1,n$.
\end{lemma}

The same Jantzen filtration arguments as in \cite[section 5]{BM1} lead then to a sharpened version of \cite[Theorem 5.3]{BM1}, which is the $p$-adic analogue of Vogan's signature theorem \cite{Vo}.

\begin{theorem}[cf. {\cite[Theorem 5.3]{BM1}}]\label{t:sig}  Let $\{V_i\}_{1\le i\le n}$ be the fixed basis of $\bar R^0(G)_\rig$ constructed above, consisting of tempered representations. Let $(\pi,V)$ be an arbitrary irreducible hermitian $G$-representation in $\C C_I(G).$ There exist integers $a_i,b_i$ such that, for every $K\supset I$, the $K$-signature character of $V$ equals
\begin{equation}
\Sigma_K(V)=(\sum_{i=1}^n a_i \theta_K(V_i),\sum_{i=1}^n b_i \theta_K(V_i)).
\end{equation}
\end{theorem}

We emphasize that in Theorem \ref{t:sig}, the integers $a_i,b_i$ do not depend on the choice of $K$, but only on $V$.

\subsection{}
The analogous definitions to $K$-character and $K$-signature exist for $\CH$-modules. Let $\CH_K=C^\infty(I\backslash K/I)$ be the subalgebra of $\CH$ consisting of functions whose support is in $K$. If $(\sigma,U_\sigma)$ is a $K$-type, then $(\sigma,U_\sigma^I)$ is an $\CH_K$-module.
If $V^I$ is a simple $\CH$-module, then the $\CH_K$-character of $V^I$ is
\begin{equation}
\theta_{\CH_K}(V^I)=\sum_{\sigma\in \widehat K,U_\sigma^I\neq 0} m_\sigma\sigma,
\end{equation}
and the $\CH_K$-signature character is
\begin{equation}
\Sigma_{\CH_K}(V^I)=(\sum_{\sigma\in\widehat K,U_\sigma^I\neq 0} p_\sigma \sigma, \sum_{\sigma\in\widehat K,U_\sigma^I\neq 0} q_\sigma \sigma).
\end{equation}
Theorem \ref{t:sig} implies that the $\CH_K$-signature character of $V^I$ is
\begin{equation}\label{e:sig-H}
\Sigma_{\CH_K}(V^I)=(\sum_{i=1}^n a_i \theta_{\CH_K}(V_i^I),\sum_{i=1}^n b_i \theta_{\CH_K}(V_i^I)),
\end{equation}
with the same integers $a_i,b_i$ as in Theorem \ref{t:sig}.

\smallskip

Now we can complete the argument. Suppose that $V^I$ is unitary as an $\CH$-module. This means that the negative part of the $K$-signature must be $0$, for all $K.$ From (\ref{e:sig-H}), this means that
\begin{equation}
\sum_{i=1}^n b_i \theta_{\C H_K}(V_i^I)=0,
\end{equation}
for all $K$. If $T_\co$ is a basis element of $\bar\CH^\rig$ as in Theorem \ref{t:basis}, there exists $K\supset I$ such that $T_\co$ is the delta function supported at a double coset $I\wti wI,$ $\wti w\in K.$ Therefore $$\sum_{i=1}^n b_i\tr(T_\co, V_i^I)=0$$ for all basis elements $T_\co$ of the rigid cocenter $\bar \CH^\rig.$ Recall that $\{V_i^I\}$ form a basis of $\bar R(\CH)_\rig$. But then Theorem \ref{t:rigid-pairing} implies that they are linearly independent over $\bar \CH^\rig$, and therefore $b_i=0$ for all $i=1,n$. 

Finally, this means that the negative part of the signature $\Sigma_K(V)$ is $0$ in Theorem \ref{t:sig}, and so, $V$ is unitary as a $G$-representation. This completes the implication:
$$\text{if }V^I \text{ is a unitary $\CH$-module, then } V\text{ is a unitary $G$-representation}.$$



\begin{thebibliography}{AAA1}
\bibitem[AM]{AM}
S.~Ariki and A.~Mathas, \emph{The number of simple modules of the Hecke algebras of type
$G(r, 1, n)$}, Math. Z. 233 (2000), 601--623.

\bibitem[Ar]{Ar}
J.~Arthur,
\emph{On elliptic tempered characters},
Acta Math. {\bf 171}, No.1, 73--138 (1993).

\bibitem[BB]{BB}
P.~Blanc, J.-L.~Brylinski,
\emph{Cyclic homology and the Selberg principle}, J. Func. Anal.
{\bf 109} (1992), 289--330.

\bibitem[BC]{BC}
D.~Barbasch, D.~Ciubotaru,
\emph{Unitary equivalences for reductive $p$-adic groups}, Amer.~Jour.~Math. {\bf 135} (2013), no. 6, 1633--1674.

\bibitem[BM1]{BM1}
D.~Barbasch, A.~Moy,
\emph{A unitarity criterion for $p$-adic groups}, Invent. Math. {\bf 98} (1989), no. 1, 19--37.

\bibitem[BM2]{BM2}
D.~Barbasch, A.~Moy,
\emph{Reduction to real infinitesimal character in affine Hecke
  algebras}, J.~Amer.~Math.~Soc. {\bf 6} (1993), 611--635.

\bibitem[BM3]{BM3}
D.~Barbasch, A.~Moy,
\emph{Unitary spherical spectrum for p-adic classical groups}, Acta Appl. Math. {\bf 44} (1996),
1--37. 

\bibitem[Be]{Be}
J.~Bernstein,
\emph{Representations of reductive $p$-adic groups}, lectures notes by K. Rummelhart, Harvard University, 1992.

\bibitem[BDK]{BDK}
J.~Bernstein, P.~Deligne, D.~Kazhdan,
\emph{Trace Paley-Wiener theorem for reductive $p$-adic groups},
J. d'Analyse Math. {\bf 47} (1986), 180--192.

\bibitem[Bo]{Bo}
A.~Borel,
\emph{Admissible representations of a semi-simple group over a local field with vectors fixed under
an Iwahori subgroup}, Invent. Math. {\bf 35} (1976), 233--259.

\bibitem[BK]{BK}
C.~Bushnell, P.~Kutzko,
{\emph Smooth representations of reductive p-adic groups: structure theory via types}, Proc. London
Math. Soc. (3) {\bf 77} (1998), no. 3, 582--634.

\bibitem[Be]{Bez}
R.~Bezrukavnikov,
\emph{Homology properties of representations of p-adic groups related to geometry of the group at infinity}, Ph.D. thesis, University of Tel Aviv (1998).

\bibitem[BHK]{BHK}
C.~Bushnell, G.~Henniart, P.~Kutzko,
\emph{Types and explicit Plancherel formulae for reductive
p-adic groups}, On Certain L-Functions, Clay Math. Proc. {\bf13}, Amer. Math. Soc., Providence,
RI, 2011, 55--80.

\bibitem[CG]{CG},
N. Chriss and V. Ginzburg, \emph{Representation Theory and Complex Geometry}, Birkh\"auser, Boston, 1997.

\bibitem[CH]{CH}
D.~Ciubotaru, X.~He,
\emph{The cocenter of graded affine Hecke algebras and the density theorem}, \texttt{arXiv:1208.0914}, preprint.

\bibitem[Da]{Da}
J.-F.~Dat,
\emph{On the $K_0$ of a $p$-adic group}, Invent. Math. {\bf 140} (2000), no. 1, 171--226. 

\bibitem[DP]{DP}
R.~Dipper, G.~James,
\emph{Representations of Hecke algebra of type $B_n$},
J.~Algebra {\bf 146} (1992), 454--481.

\bibitem[Fl]{Fl}
Y.~Flicker, 
\emph{Bernstein's isomorphism and good forms},  K-Theory and Algebraic Geometry: Connections with Quadratic Forms and Division Algebras, 1992 Summer Research Institute; Proc. Symp. Pure Math. {\bf 58} II (1995), 171--196, AMS, Providence RI.

\bibitem[GP]{GP}
M.~Geck, G.~Pfeiffer,
\emph{Characters of finite Coxeter groups and Iwahori-Hecke algebras}, London Mathematical Society Monographs. New Series, vol. {\bf 21}, The Clarendon Press Oxford University Press, New York, 2000. 


\bibitem[He]{He99}
X.~He, \emph{Geometric and homological properties of affine Deligne-Lusztig varieties}, Ann. Math. \textbf{179} (2014), 367--404.  

\bibitem[HN1]{HN}
X.~He, S.~Nie, \emph{Minimal length elements of extended affine Weyl group}, arXiv:1112.0824. 

\bibitem[HN2]{HN2}
X.~He, S.~Nie, \emph{$P$-alcoves, parabolic subalgebras and cocenters of affine Hecke algebras}, arXiv:1310.3940.

\bibitem[Hei]{He}
V.~Heiermann,
\emph{Op\'erateurs d'entrelacement et alg\`ebres de Hecke avec param\`etres d'un groupe r\'eductif
p-adique: le cas des groupes classiques}, Selecta Math. (N.S.) {\bf 17} (2011), no. 3, 713--756.

\bibitem[HM]{HM}
R.~Howe, with the collaboration of A. Moy, \emph{Harish-Chandra homomorphisms for p-adic groups},
CBMS Reg. Conf. Ser. Math. {\bf 59}, Amer. Math. Soc., Providence, RI, 1985.

\bibitem[IM]{IM}
N.~Iwahori, H.~Matsumoto,
\emph{On some Bruhat decomposition and the structure of the Hecke rings of
p-adic Chevalley groups}, Inst. Hautes \'Etudes Sci. Publ. Math.  {\bf 25} (1965), 5--48.

\bibitem[Kat]{Kat}
S.~Kato,
\emph{An exotic Deligne-Langlands correspondence for symplectic
groups}, Duke Math. J. {\bf 148} (2009), no. 2, 305--371.


\bibitem[Kaz1]{Kaz}
D.~Kazhdan, \emph{Cuspidal geometry of $p$-adic groups}, J. Analyse Math. {\bf 47} (1986), 1--36.

\bibitem[Kaz2]{Ka}
D.~Kazhdan,
\emph{Representations of groups over close local fields}, J. Analyse Math.
{\bf 47} (1986), 175--179. 

\bibitem[KL]{KL}
D.~Kazhdan, G.~Lusztig,
\emph{Proof of the Deligne-Langlands conjecture for Hecke algebras}, Invent.
Math. {\bf 87} (1987), no. 1, 153--215.


\bibitem[Ko1]{Ko1} R.~Kottwitz, \emph{Isocrystals with additional
structure}, Compositio Math.  \textbf{56} (1985), 201--220.

\bibitem[Ko2]{Ko2} R.~Kottwitz, \emph{Isocrystals with
additional structure. {II}}, Compositio Math. \textbf{109} (1997), 255--339.

\bibitem[Lu1]{L1}
G.~Lusztig,
\emph{Affine Hecke algebras and their graded version},
J. Amer. Math. Soc. 2 (1989), 599--635.

\bibitem[Lu2]{L2}
\emph{Cuspidal local systems and graded algebras II}, Representations
of groups (Banff, AB, 1994), Amer. Math. Soc., Providence, 1995,
217--275.  

\bibitem[Mi]{Mi}
V. Miemietz, \emph{On representations of affine Hecke algebras of type $B$}, Algebr. Represent. Theory \textbf{11} (2008), no. 4, 369--405. 

\bibitem[Op]{Op}
E.~Opdam,
\emph{On the spectral decomposition of affine Hecke algebras}, J. Inst. Math. Jussieu 3 (2004), no. 4, 531--648.

\bibitem[OS1]{OS1}
E.~Opdam, M.~Solleveld, \emph{Discrete series characters for affine Hecke algebras and their formal degrees}, Acta Math. {\bf 205} (2010), no. 1, 105--187.

\bibitem[OS2]{OS2}
E.~Opdam, M.~Solleveld,
\emph{Homological algebra for affine Hecke algebras}, Adv. Math. {\bf 220} (2009), no. 5, 1549--1601.

\bibitem[RR]{RR}
A.~Ram, J.~Ramagge,
\emph{Affine Hecke algebras, cyclotomic Hecke algebras and Clifford theory},
Birkh\"auser, Trends in Mathematics (2003), 428--466.

\bibitem[Re]{Re}
M.~Reeder,
\emph{Euler-Poincar\' e pairings and elliptic representations of Weyl groups and $p$-adic groups},
Compos. Math. {\bf 129}, No. 2, 149--181 (2001). 

\bibitem[SS]{SS}
P.~Schneider, U.~Stuhler,
\emph{Representation theory and sheaves on the Bruhat-Tits building},
Publ.~Math.~Inst.~Hautes~\'Etud.~Sci. {\bf 85}, 97--191 (1997).

\bibitem[SVV]{SVV}
P.~Shan, M.~Varagnolo, E.~Vasserot,
\emph{Canonical bases and affine Hecke algebras of type $D$},
Adv.~Math. {\bf 227} (2011), no. 1, 267--291.

\bibitem[So]{So}
M.~Solleveld,
\emph{Hochschild homology of affine Hecke algebras}, J.~Algebra {\bf 384} (2013), 1--35.

\bibitem[Vi]{Vi}
M.-V. Vign\'eras, \emph{Modular representations of p-adic groups and of affine Hecke algebras},
Proc. of Inter. Congress. Math., Beijing 2002, Vol. 2, pp. 667--677, Higher Eduction Press,
2002.

\bibitem[VV]{VV}
M.~Varagnolo, E.~Vasserot,
\emph{Canonical bases and affine Hecke algebras of type $B$},
Invent.~Math. {\bf 185} (2011), no. 3, 593--693.

\bibitem[Vo]{Vo}
D.~A.~Vogan~Jr., 
\emph{Unitarizability of certain series of representations}, Ann.~of~Math. (2) {\bf 120} (1984), no. 1,
141--187.

\bibitem[Xi1]{Xi}
N. Xi, \emph{Representations of affine Hecke algebras and based rings of affine Weyl groups}, J. Amer. Math. Soc. {\bf 20} (2007), 211--217.

\bibitem[Xi2]{Xi2}
N. Xi, \emph{Representations of affine Hecke algebras of type $G_2$}, Acta Math. Sci. Ser. B Engl. Ed. {\bf 29} (3) (2009), 515--526. 

\end{thebibliography}
\end{document}